\documentclass{birkjour}

\usepackage{amssymb}
\usepackage{physics}
\usepackage{bm}

 \newtheorem{thm}{Theorem}[section]
 \newtheorem{cor}[thm]{Corollary}
 
 \newtheorem{prop}[thm]{Proposition}
 \theoremstyle{definition}
 \newtheorem{defn}[thm]{Definition}
 \theoremstyle{remark}
 \newtheorem{rem}[thm]{Remark}
 
 \numberwithin{equation}{section}

\begin{document}

%
%
%
%
%
%
%
%
%

\title[Global Well-Posedness for Eddy-Mean Vorticity Equations]{Global Well-Posedness for\\ Eddy-Mean Vorticity Equations on  $\mathbb{T}^2$}


\author[Y. Cacchio']{Yuri Cacchio'}

\address{%
SBAI Sapienza\\
Via Antonio Scarpa, 14\\
00161 Roma\\
Italy}

\email{yuri.cacchio@uniroma1.it}

\thanks{The author was supported in part by Sapienza "Giovani Ricercatori" Grant DR n.1607 }


\date{\today}

\begin{abstract}
We consider the two-dimensional, $\beta$-plane, eddy-mean vorticity equations for an  incompressible flow, where  the zonally averaged flow varies on scales much larger than the perturbation. We prove global existence and uniqueness of the solution to the equations on periodic settings.
\end{abstract}

\maketitle
\section{Introduction}

Atmospheric and oceanic flows on a rotating sphere are often modeled using equations for a single fluid layer of constant density as a consequence of their large horizontal extent compared with their depth. However, on large scale these flows are strongly influenced not only by  rotation but also by stratification \cite{Batchelor,Cope,Bertozzi,McWilliams,Pope}.

In three-dimensional stratified configurations turbulence naturally arises as a result of dynamical instabilities such as baroclinic instability, however, there is no such underlying instability in the two-dimensional case \cite{Cope}. Consequently, it is common to artificially force two-dimensional turbulence through an exogenous, statistically homogeneous, forcing function. A state of equilibrium can then be achieved by the inclusion of dissipation terms \cite{Liu, ScottPolv}.

In this article we take into account only the rotation by neglecting stratification in a single layer model where we incorporate the effects of planetary rotation by adopting a beta-plane approximation, which is a simple device used to represent the latitudinal variation in the vertical component of the planetary rotation \cite{holton,pedlosky,vallis1}.

Keeping this in mind, let us start by considering the two-dimensional beta-plane vorticity equation
\begin{equation}\label{vorticityeq}
    \partial_t\zeta+J(\psi,\zeta+\beta y)=\nu\nabla^2\zeta,
\end{equation}
where $J(A,B)=A_xB_y-A_yB_x$, which is completely determined by a single dynamical variable, the stream function $\psi$, since the vorticity $\zeta=\nabla^2 \psi$.

Equation \eqref{vorticityeq} has been used in a wide range of studies investigating large scale planetary flows in double periodic geometry \cite{Cost,Danilov,Smith,vallis4} and on the sphere \cite{huang,nozawa,scott}. 

Introducing the eddy-mean decomposition, sometimes called  Reynolds decomposition, denoted by a prime and an overbar respectively,
\begin{equation}\label{average}
    f(t,x,y)=\overline{f}(t,y)+f'(t,x,y),
\end{equation}
we can split equation \eqref{vorticityeq} into mean and fluctuating components obtaining the following initial value problem

\begin{equation}\label{generalProblem}
    \left\{\begin{array}{ll}
         &\partial_t \zeta'+C_1\partial_x\nabla^{-2}\zeta'-\nu\nabla^2\zeta'=F(\overline{u},\zeta') \\
         &\partial_t\overline{u}-\nu \nabla^2\overline{u}=G(\zeta')\\
         &\zeta'(0,x,y)=\zeta'_0(x,y) \\
         &\overline{u}(0,y)=\overline{u}_0(y)
    \end{array}
    \right.
\end{equation}
with
\begin{align*}
    F(\overline{u},\zeta')&:=(\partial_y\nabla^{-2}\zeta')\zeta'_x-(\partial_x\nabla^{-2}\zeta')\zeta'_y-\overline{u}\zeta'_x+\partial_y\overline{[(\partial_x\nabla^{-2}\zeta')\zeta']},\\
    G(\zeta')&:=\partial_y\overline{[\partial_x\nabla^{-2}\zeta'\partial_y\nabla^{-2}\zeta']}.
\end{align*}
We denote by $\overline{u}=\overline{u}(t,y)$ the zonal mean velocity, also called jet velocity profile, and by $\zeta'=\zeta'(t,x,y)$ the eddy-vorticity; the parameter $\nu>0$ is called kinematic viscosity.
Here, the system of equations \eqref{generalProblem} is posed on a periodic box (periodic boundary condition) $\mathbb{T}^2_l=[0,l)^2$ of size $l>0$.

In the next section we see a more detailed description of the model \eqref{vorticityeq} and how to get system \eqref{generalProblem} from equation \eqref{vorticityeq}.

The goal of the paper is to prove existence and uniqueness of the solution of problem \eqref{generalProblem} in the periodic setting, which allows us to determine the Reynolds stress as shown in Section 2. The latter quantity turns out to be crucial in fluid dynamics since it is the component of the total stress tensor in a fluid to account for turbulent fluctuations in fluid momentum. Moreover, as written in \cite[p. 414]{vallis1}, the "closure problem" of turbulence may be thought of as finding a representation of such Reynolds stress terms in terms of mean flow quantities, which seems to be critical without introducing physical assumptions not directly deducible from the equations of motion themselves.

In light of that we prove the following global well-posedness result:
\begin{thm}\label{theoprincipale}
If $\overline{u}_0\in H^s(\mathbb{T}_l)$ and $\zeta'_0\in H^s(\mathbb{T}^2_l)$, $s\geq0$, then there exists a unique global solution $(\overline{u},\zeta')\in L^{\infty}\left((0,+\infty);\ H^s(\mathbb{T}_l)\cross H^s(\mathbb{T}^2_l)\right)$ of the initial value problem $\eqref{generalProblem}$.
\end{thm}

The proof of the above theorem is first obtained locally in time via a contraction method, and then extended using an iteration method based on the conservation of the $L^2$-norm. The estimates used in the argument are reminiscent of those used to study periodic dispersive initial value problems 
and  inspired by the pioneering works of Bourgain \cite{Bourgain1,Bourgain2,Bourgain4} and Kenig, Ponce and Vega \cite{vega1,vega2}. Relevant results in periodic settings can also be found in \cite{Bourgain3,Staffilani}. We also refer the interested reader to Tao's notes \cite{TaoNote} for an introduction to this method, where he shows local well-posedness of the Navier-Stokes equations.

Therefore, we proceed as follows:
\begin{itemize}
    \item [1)]\textbf{Duhamel's principle.} Since in Theorem \ref {theoprincipale} we assume very low regularity it is clear that the initial value problem \eqref{generalProblem} needs to be interpreted in an appropriate manner. To this end the first step is to use the Duhamel's principle so that the solution to problem \eqref{generalProblem} can be interpreted as the solution of the integral system
    \begin{align}\label{duha}
        (\overline{u},\zeta')=\bigg(&e^{\nu t\partial_{yy}}\overline{u}_0+\int_0^t e^{(\nu\partial_{yy})(t-t')}G(\zeta')\ dt',\\
        &e^{(\nu\nabla^2-\partial_x\nabla^{-2})t}\left.\zeta'_0+\int_0^t e^{(\nu\nabla^2-\partial_x\nabla^{-2})(t-t')}F(\overline{u},\zeta')\ dt'\right)\notag
    \end{align}
    \item[2)] \textbf{Upper bounds.} The second step is to define the  functional  
     \begin{eqnarray*}
        \Phi(\overline{u},\zeta'):=\bigg(&e^{\nu t\partial_{yy}}\overline{u}_0+\int_0^t e^{(\nu\partial_{yy})(t-t')}G(\zeta')\ dt',\\
        &e^{(\nu\nabla^2-\partial_x\nabla^{-2})t}\left.\zeta'_0+\int_0^t e^{(\nu\nabla^2-\partial_x\nabla^{-2})(t-t')}F(\overline{u},\zeta')\ dt'\right)
    \end{eqnarray*}   
  and prove that 
     \begin{equation}\label{stimenormasoluzione}
        \norm{ \Phi(\overline{u},\zeta')}_{L^\infty_{[0,\delta]}(H^s(\mathbb{T}_l)\cross H^s(\mathbb{T}_l^2))}\leq C(\norm{\overline{u}_0}_{H^s(\mathbb{T}_l)}+\norm{\zeta'_0}_{H^s(\mathbb{T}_l^2)}),
    \end{equation}
    $\forall \ s \geq0$, on a small enough time interval $[0,\delta]$.  Similarly one also proves that $\Phi$ is a contraction. In this way we 
    have that 
    \begin{equation*}
        \Phi:B(0,R)\to B(0,R)
    \end{equation*}
    on a suitable ball of radius $R>0$ and contraction.
       \item[3)]\textbf{Local Well-Posedness.} The next step is to use the fixed point theorem on the ball defined in step 2) in order to prove existence and uniqueness of the solution on $B(0,R)$. After that, due to the estimate \eqref{stimenormasoluzione}, we extend this solution from the ball over the whole space $L^\infty_{[0,\delta]}(H^s(\mathbb{T}_l)\cross H^s(\mathbb{T}_l^2))$, with $s\geq0$.
    \item[4)]\textbf{Global Well-Posedness.} Finally, since the small time $\delta$ in Step 2 depends only on the $L^2$-norms of $\overline{u}$ and $\zeta'$, and these can be proved to be  bounded by $L^2$-norms of initial data, i.e.,
    \begin{equation}
        \norm{\overline{u}}_{L^2_y(\mathbb{T}_l)}\leq C \norm{\overline{u}_0}_{L^2_y(\mathbb{T}_l)} \text{ and } \norm{\zeta'}_{L^2_{xy}(\mathbb{T}_l^2)}\leq C \norm{\zeta'_0}_{L^2_{xy}(\mathbb{T}_l^2)},
    \end{equation}
    for some universal constant $C$, and we can extend the solution by covering the whole time interval $[0,\infty)$ by iteration.
\end{itemize}

The organization of the paper is as follows. In Section 2, we state assumptions, certain terminology and we derive from the vorticity equation \eqref{vorticityeq} the system of equations \eqref{generalProblem}.
In Section  3, as mentioned in Steps 2 and 3 above, we prove inequality \eqref{stimenormasoluzione}, we properly define the functional $\Phi$, and after proving that it is a contraction we obtain local well-posedness on a small enough time interval.
Finally, we extend the solution by iteration to get the global well-posedness result of \eqref{generalProblem} in the last section.

\section{The Eddy-Mean Decomposition and Assumptions}
To establish the notation used in this paper, we start with the incompressible 2D Navier-Stokes equations \cite{Batchelor,kuk,salmon,sriniv},
\begin{equation}\label{NS}
    \left\{\begin{array}{rl}
      \partial_t\bm{u}+(\bm{u}\cdot \nabla)\bm{u}+f\Vec{\bm{z}}\cross \bm{u}&=\nu\nabla^2\bm{u}-\nabla p\\
    \div{\bm{u}}&=0
    \end{array}
    \right.
\end{equation}
where we include the effect of the planetary rotation through the Coriolis force $f$ as a tool to understand various geophysical flows \cite{mustafa,gallagher2,gallagher1,rhines}. Here $\bm{u}=(u,v)$ and $p$ are unknown velocity field and pressure, $\nu>0$ is the kinematic viscosity and $(\bm{u}\cdot\nabla)$ stands for the differential operator $u\partial_x+v\partial_y$. 

For the purpose of this discussion, we do an important approximation, so-called $\beta$-$plane$,  which captures the most important dynamical effects of sphericity, without the complicating geometric effects, which are not essential to describe many phenomena \cite{vallis1}. Since the magnitude of the vertical component of rotation varies with latitude, we can approximate this effect by allowing the effective rotation vector to vary. Thus, Taylor-expanding the Coriolis parameter around a latitude $\varTheta_0$, for small variation in latitude, we have \cite{holton,pedlosky,salmon} 
\begin{equation}
    f=2\Omega\sin\varTheta\approx 2\Omega\sin{\varTheta_0}+2\Omega(\varTheta-\varTheta_0)\cos{\varTheta_0},
\end{equation}
where $\varTheta$ is the latitude and $\Omega$ is the angular velocity of the sphere. Then, on the tangent plane we may mimic this by allowing the Coriolis parameter to vary as
\begin{equation}
    f=f_0+\beta y
\end{equation}
where $f_0=2\Omega\sin{\varTheta_0}$, $\beta=\partial f/\partial y=(2\Omega\cos{\varTheta_0})/a$ and $a$ is the radius of the planet.

Moreover, since in 2D flows the velocity field has two component which depend on two physical space coordinates and time, $\bm{u}=(u(t,x,y),v(t,x,y),0)$, then the vorticity field $\bm{\zeta}$ has only one non-zero component,
\begin{equation}\label{vorticita}
    \bm{\zeta}:=\nabla\cross \bm{u}=(0,0,\zeta(t,x,y))=(0,0,v_x-u_y).
\end{equation}
This component satisfies the $\beta$-plane vorticity equation, 
\begin{equation}\label{equazionevorticita}
    \zeta_t+u\zeta_x+v\zeta_y+\beta v=\nu\nabla^2\zeta,
\end{equation}
which is obtained by taking the curl from \eqref{NS}.

For 2D incompressible flows, we can introduce a representation of the velocity field in terms of the stream function $\psi(t,x,y)$ as in  \cite{Nazarenko},
\begin{align}
    (u,v)&=(-\psi_y,\psi_x),\label{velocitarelaz}\\
    \zeta=\nabla^2\psi&=\psi_{xx}+\psi_{yy}\label{vorticitarelaz}.
\end{align}
In addition, since we work in a periodic frame, we can use the eddy-mean decomposition \eqref{average} in \eqref{equazionevorticita}, where 
\begin{equation*}
    \overline{f}=\frac{1}{l}\int_0^l f \ dx
\end{equation*}
and get the zonal mean momentum equation
\begin{equation}\label{zonalmomentumequation}
    \partial_t\overline{u}+\partial_y(\overline{u'v'})=\nu \nabla^2\overline{u},
\end{equation}
and the eddy vorticity equation
\begin{equation}\label{eddyvorticityequation}
    \partial_t\zeta'+\overline{u}\zeta'_x+(\beta-\overline{u}_{yy})v'=\nu\nabla^2\zeta'+\partial_y(\overline{v'\zeta'})-u'\zeta'_x-v'\zeta_y'.
\end{equation}
The last hypothesis we make is to assume constant $\overline{u}_{yy}$, that is we assume that the zonally averaged flow varies on scales much larger than the perturbation \cite{Cost,sriniv}.

Finally, from \eqref{average}, \eqref{velocitarelaz} and \eqref{vorticitarelaz} we obtain 
\begin{align}\label{Reystress}
     u'&=-\partial_y\nabla^{-2}\zeta';\notag\\
    v'&=\partial_x\nabla^{-2}\zeta',
\end{align}
and we derive the initial value problem \eqref{generalProblem} by replacing relations \eqref{Reystress} in \eqref{zonalmomentumequation} and \eqref{eddyvorticityequation}.

As mentioned in the introduction, proving well posedness for \eqref{generalProblem} will allow us to derive the Reynolds stress using relations \eqref{Reystress}, since it is defined as $\overline{u'v'}$ \cite{McWilliams}.

Before we address the computations of well posedness, with the purpose of fixing the functions space on which we work, we show that the average of $\overline{u}$ and $\zeta'$ are both conserved. In fact, for $\overline{u}$ we have 
\begin{align*}
    \partial_t\int_{0}^l\overline{u}\ dy&=\int_{0}^l\partial_t\overline{u}\ dy\\
    &=\int_{0}^l\nu \partial^2_y\overline{u}+\partial_y\overline{[\partial_x\nabla^{-2}\zeta'\partial_y\nabla^{-2}\zeta']}\ dy\\
    &=\nu\left[\partial_y\overline{u}\right]^l_{0}+\left[\overline{[\partial_x\nabla^{-2}\zeta'\partial_y\nabla^{-2}\zeta']}\right]^l_{0}=0
\end{align*}
which vanishes by using periodic boundary conditions. On the other hand, for $\zeta'$, recalling that $\zeta=\zeta'+\overline{\zeta}$ with $\overline{\zeta}=\frac{1}{l}\int_0^l \zeta \ dx$, we have
\begin{equation*}
     \int_0^l\int_0^l\zeta'\ dxdy= \int_0^l\int_0^l\zeta\ dxdy- \int_0^l\int_0^l\overline{\zeta}\ dxdy=l\int_0^l\overline{\zeta}\ dy-l\int_0^l\overline{\zeta}\ dy=0.
\end{equation*}
Then, 
\begin{equation}\label{condfour2}
    \widehat{\zeta'}(0)= \int_0^l\int_0^l\zeta'(x,y)e^{i0\cdot(x,y)}\ dxdy= \int_0^l\int_0^l \zeta'(x,y)\ dxdy=0,
\end{equation}
and if we define
\begin{equation*}
    c_0=\int_{0}^l\overline{u}\ dy,
\end{equation*}
we have,
\begin{equation}\label{condfour1}
    c_0=\int_{0}^l\overline{u}\ dy=\int_{0}^l\overline{u}(y)e^{i0\cdot y}\ dx=\widehat{\overline{u}}(0).
\end{equation}
We set 
\begin{equation*}
    \mu:=\overline{u}-c_0
\end{equation*}
so that 
\begin{equation}\label{condizionefour}
    \widehat{\mu}(0)=0.
\end{equation}
Moreover, using relation \eqref{zonalmomentumequation},
\begin{align}\label{riscritto}
    \partial_t \mu=\partial_t\overline{u}&=\nu \nabla^2\overline{u}+\partial_y\overline{[\partial_x\nabla^{-2}\zeta'\partial_y\nabla^{-2}\zeta']}\notag\\
    &=\nu \nabla^2 \mu+\partial_y\overline{[\partial_x\nabla^{-2}\zeta'\partial_y\nabla^{-2}\zeta']}.
\end{align}
Then $\mu$ satisfies equation \eqref{zonalmomentumequation} with the additional condition \eqref{condizionefour}.
Due to the above remarks, we study the initial value problem \eqref{generalProblem} equivalently in the following functions spaces.

\begin{defn}
    \begin{equation*}
    X^{s}:=\left\{(f,g)\in L^\infty\left((0,\infty);\ H^s(\mathbb{T}_l)\cross H^s(\mathbb{T}^2_l)\right); \widehat{f}(0)=0,\ \widehat{g}(0)=0\right\},
    \end{equation*}
    and in a similar way
    \begin{equation*}
        X^{s,\delta}:=\left\{(f,g)\in L^\infty\left([0,\delta];\ H^s(\mathbb{T}_l)\cross H^s(\mathbb{T}^2_l)\right); \widehat{f}(0)=0,\ \widehat{g}(0)=0\right\},
    \end{equation*}
    equipped with norm
    \begin{align*}
        \norm{(f,g)}_{X^s}&=\norm{f}_{L^\infty_{[0,\infty)} H^s(\mathbb{T}_l)}+\norm{g}_{L^\infty_{[0,\infty)} H^s(\mathbb{T}_l^2)},\\
        \norm{(f,g)}_{X^{s,\delta}}&=\norm{f}_{L^\infty_{[0,\delta]} H^s(\mathbb{T}_l)}+\norm{g}_{L^\infty_{[0,\delta]} H^s(\mathbb{T}_l^2)},
    \end{align*}
    respectively.
\end{defn}
Via  Plancherel's theorem we express the $\norm{\cdot}_{H^s}$-norm in the Fourier modes, namely 
\begin{equation*}
    \norm{f}^2_{H^s}=\sum_k \langle k\rangle^{2s}|\widehat{f}|^2,
\end{equation*}
where $\langle k\rangle:=(1+|k|^2)^{\frac{1}{2}}$ is the Japanese bracket.

\textbf{Notation 1}: 
Throughout the paper we use $A\lesssim B$ to denote an estimate of the form $A \leq CB$ for some absolute constant $C$. If $A \lesssim B$ and $B \lesssim A$ we write $A \sim B$.

\textbf{Notation 2}: 
From here on we use the following notations,
\begin{align*}
    \norm{\cdot}_{L^\infty_{[0,\infty)}}&=\norm{\cdot}_{L^\infty_t}\\
    \norm{\cdot}_{L^\infty_{[0,\delta]}}&=\norm{\cdot}_{L^\infty_{\delta}}\\
    \norm{\cdot}_{H^s(\mathbb{T}^2_l)}&=\norm{\cdot}_{H^s_{xy}}.
\end{align*}

\section{Local Well-Posedness}
In this section we prove well posedness of \eqref{generalProblem} on the space $X^{s,\delta}$. In order to do this, we first set the initial data of the problem. We fix 
\begin{equation}
    \left(\zeta'_0(x,y), \overline{u}_0(y)\right)\in H^s(\mathbb{T}^2_l)\cross H^s(\mathbb{T}_l).
\end{equation}
Due to \eqref{riscritto}, we equivalently rewrite \eqref{generalProblem} in the following form 
\begin{equation}\label{sistemariscritto}\left\{
    \begin{array}{rl}
     \partial_t \mu-\nu\partial_{yy} \mu &=G(\gamma), \\
    \partial_t \gamma+C_1\partial_x\nabla^{-2}\gamma-\nu\nabla^2\gamma&=F(\mu,\gamma), \\
    \mu(0,y)&=\mu_0(y)\in H^s(\mathbb{T}_l),\\
    \gamma(0,x,y)&=\gamma_0(x,y)\in H^s(\mathbb{T}^2_l),
\end{array}
\right.
\end{equation}
where we define,
\begin{align}\label{notatmugamma}
    \mu(t,y)&:=\overline{u}(t,y)-c_0,\notag\\
    \gamma(t,x,y)&:=\zeta'(t,x,y),\\
    G(\gamma)&:=\partial_y\overline{[\partial_x\nabla^{-2}\gamma\partial_y\nabla^{-2}\gamma]},\notag\\
    F(\mu,\gamma)&:=(\partial_y\nabla^{-2}\gamma)\gamma_x-(\partial_x\nabla^{-2}\gamma)\gamma_y-\mu\gamma_x+c_0\gamma_x+\partial_y\overline{[(\partial_x\nabla^{-2}\gamma)\gamma]}.\notag
\end{align}
We now use the Duhamel's principle so that the solution to problem \eqref{sistemariscritto} can be interpreted as the solution of the integral system, 
\begin{align}
        (\mu,\gamma)=\bigg(&e^{\nu t\partial_{yy}}\mu_0+\int_0^t e^{(\nu\partial_{yy})(t-t')}G(\gamma)\ dt',\\
        &e^{\tilde{D}^2t}\left.\gamma_0+\int_0^t e^{\tilde{D}^2(t-t')}F(\mu,\gamma)\ dt'\right)\notag
\end{align}
with $\tilde{D}^2=\nu\nabla^2-C_1\partial_x\nabla^{-2}$. It is not restrictive to assume $C_1=1$ as we will see below.

\subsection{ Bounds and estimates}
We proceed by deriving estimates of the norm of the various terms of the Duhamel expressions above.
\subsubsection{Zonal-Mean Momentum Equation Estimate}
Let us start with the equation for $\mu$.

\begin{prop}\label{prop1u}
Let $t\in [0,\delta]$, $\delta\in \mathbb{R}_{+}$, $s\geq0$, $\alpha\in(\frac{1}{2},1)$. Then,
\begin{equation*}
    \norm{\int_0^t e^{(t-t')\nu\partial_{yy}}G(\gamma)\ dt'}_{L^\infty_\delta H_{y}^s}\lesssim \delta^{1-\alpha}\norm{\gamma}^2_{L^\infty_\delta H^s_{xy}}.
    \end{equation*}
\end{prop}

\begin{proof}
We recall that,
\begin{align*}
      \norm{\int_0^t e^{(t-t')\nu\partial_{yy}}G(\gamma) dt'}_{L^\infty_\delta H_{y}^s} \\ =\sup_{t\in[0,\delta]}&\left(\sum_{k_2} \left|\int_0^t e^{-(t-t')\nu k^2_2}\widehat{G(\gamma)}(k_2)dt'\right|^2\langle k_2\rangle^{2s}\right)^{\frac{1}{2}}
\end{align*}
where,
\begin{align}
    \widehat{G(\gamma)}(k_2)&=\mathfrak{F}_k\left(\partial_y\overline{[\partial_x\nabla^{-2}\gamma\partial_y\nabla^{-2}\gamma]}\right)\notag\\
    &=ik_2\left.\left(-i\frac{k_1}{|k|^2}\widehat{\gamma}*(-i\frac{k_2}{|k|^2}\widehat{\gamma})\right)\right|_{[0,k_2]}\notag\\
    &=\sum_{\substack{k_1=h_1+m_1=0 \\k_2=h_2+m_2\\
    m\neq0\\
    h\neq0}}i\frac{h_1m_2k_2}{|h|^2|m|^2}\widehat{\gamma}(h)\widehat{\gamma}(m).\label{nonzerocondition}
\end{align}
In fact, for a general function $f(x,y)\in H^s_{xy}(\mathbb{T}^2_l)$,
\begin{equation*}
    \overline{f(x,y)}:=\frac{1}{l}\int_0^l f(x,y)\ dx=\overline{f}(y)
\end{equation*}
and
\begin{align*}
    \widehat{\overline{f(x,y)}}(k_1,k_2)&=\int e^{iyk_2}\overline{f}(y)\ dy\\
    &=\int e^{iyk_2}\left(\frac{1}{l}\int_0^l f(x,y)\ dx\right)dy\\
    &=\frac{1}{l}\int_0^l\left(\int e^{iyk_2}f(x,y)\ dy\right)dx\\
    &=\frac{1}{l}\int_0^l\widehat{f}(0,k_2)\ dx\\
    &=\widehat{f}(0,k_2)\frac{1}{l}\int_0^l dx=\widehat{f}(0,k_2).
\end{align*}

\begin{rem}\label{remarknozero}
Due to \eqref{condfour2} and \eqref{condizionefour} we can assume $m\neq0$ and $h\neq0$ in \eqref{nonzerocondition}. Henceforth we omit this subscript.
\end{rem}
For now, we only consider
\begin{align*}
      \norm{\int_0^t e^{(t-t')\nu\partial_{yy}}G(\gamma) dt'}^2_{H_{y}^s}&\leq \int_0^t \norm{e^{(t-t')\nu\partial_{yy}}G(\gamma)}^2_{H_{y}^s} dt'\\
      &=\int_0^t \sum_{k_2}\left| e^{-(t-t')\nu k^2_2}\widehat{G(\gamma)}(k_2)\right|^2\langle k_2\rangle^{2s}dt'.
\end{align*}
Using \eqref{nonzerocondition},
\begin{align*}
      &=\int_0^t \sum_{k_2}\Bigg| e^{-(t-t')\nu k^2_2}\sum_{\substack{k_1=h_1+m_1=0 \\k_2=h_2+m_2}}i\frac{h_1m_2k_2}{|h|^2|m|^2}\widehat{\gamma}(h)\widehat{\gamma}(m)\Bigg|^2\langle k_2\rangle^{2s}dt'\\
      &\leq \int_0^t \sum_{k_2}\Bigg( e^{-(t-t')\nu k^2_2}\sum_{\substack{k_1=h_1+m_1=0 \\k_2=h_2+m_2}}\frac{|h||m||k_2|\langle k_2\rangle^{s}}{|h|^2|m|^2}\widehat{\gamma}(h)\widehat{\gamma}(m)\Bigg)^2dt'\\
      &= \int_0^t \sum_{k_2}\Bigg( \frac{e^{-(t-t')\nu k^2_2}((t-t')\nu k_2^2)^\alpha}{((t-t')\nu k_2^2)^\alpha}\sum_{\substack{k_1=h_1+m_1=0 \\k_2=h_2+m_2}}\frac{|k_2|\langle k_2\rangle^{s}}{|h||m|}\widehat{\gamma}(h)\widehat{\gamma}(m)\Bigg)^2dt'.
\end{align*}
Since
\begin{equation}\label{uniformbound1}
    e^{-(t-t')\nu k^2_2}((t-t')\nu k^2_2)^\alpha\leq C, \ \ \alpha\in [0,1],
\end{equation}
is uniformly bounded, we can focus on
\begin{equation*}
    \sum_{k_2}\Bigg(\frac{1}{(\nu k_2^2)^\alpha}\sum_{\substack{k_1=h_1+m_1=0 \\k_2=h_2+m_2}}\frac{|k_2|\langle k_2\rangle^s}{|h||m|}|\widehat{\gamma}(h)||\widehat{\gamma}(m)| \Bigg)^2,
\end{equation*}
which is equivalent to 
\begin{equation*}
    \sum_{k_2}\Bigg(\frac{1}{(\nu k_2^2)^\alpha}\sum_{\substack{h_1 \\k_2=h_2+m_2}}\frac{|k_2|\langle k_2\rangle^s}{|h||\tilde{m}|}|\widehat{\gamma}(h)||\widehat{\gamma}(\tilde{m})| \Bigg)^2
\end{equation*}
with
\begin{equation*}
    \tilde{m}=(-h_1,m_2).
\end{equation*}
By duality, we want to show
\begin{equation*}
    \sup_{\norm{g}_{l^2}\leq 1}\sum_{k_2}\sum_{\substack{h_1 \\k_2=h_2+m_2}}\frac{V(k_2,h,\tilde{m})}{\langle h\rangle^s\langle \tilde{m}\rangle^s}f_1(h)f_2(\tilde{m})g(k_2)\lesssim \norm{f_1}_{l^2}\norm{f_2}_{l^2}\norm{g}_{l^2},
\end{equation*}
where
\begin{align*}
    V(k_2,h,\tilde{m})&:=\frac{\langle k_2\rangle^s}{|k_2|^{2\alpha-1}|h||\tilde{m}|};\\
    f_1(h)&:=\langle h\rangle^s|\widehat{\gamma}(h)|;\\
    f_2(\tilde{m})&:=\langle \tilde{m}\rangle^s|\widehat{\gamma}(\tilde{m})|.
\end{align*}
We remark that if $\gamma\in H^s$ then $f_i\in l^2$ for $i=1,2$. Moreover, we have 
\begin{equation*}
    \langle k\rangle^s=\left((1+|k|^2)^{\frac{1}{2}}\right)^{s}\sim|k|^s.
\end{equation*}
Hence, 
\begin{equation*}
    \frac{V(k_2,h,\tilde{m})}{\langle h\rangle^s\langle \tilde{m}\rangle^s}\sim \frac{|k_2|^{s-2\alpha+1}}{|h|^{s+1}|\tilde{m}|^{s+1}}.
\end{equation*}\newpage
We are ready to study all possible cases as $h,m$ and $k$ varies:
\begin{itemize}
    \item [1)] $|h_2|\gg|m_2|\Rightarrow|h_2|\sim|k_2|$.
    \begin{align*}
        \frac{|k_2|^{s-2\alpha+1}}{|h|^{s+1}|\tilde{m}|^{s+1}}\sim\frac{|h_2|^{s-2\alpha+1}}{(|h_1|^2+|h_2|^2)^{\frac{s+1}{2}}(|h_1|^2+|m_2|^2)^{\frac{s+1}{2}}}.
    \end{align*}    
Since $|\tilde{m}|^{s+1}=(|h_1|^2+|m_2|^2)^{\frac{s+1}{2}}\geq 1$ and $|h_2|\neq0$ because $|h_2|\gg|m_2|$, 
    \begin{align*}
        &\leq \frac{|h_2|^{s-2\alpha+1}}{|h_2|^{s+1}}\leq \frac{1}{|h_2|^{2\alpha}}.
    \end{align*}
    Using three times Cauchy-Schwartz inequality we have 
      \begin{align*}
        &\sum_{k_2}\sum_{\substack{h_1 \\k_2=h_2+m_2}}\frac{V(k_2,h,\tilde{m})}{\langle h\rangle^s\langle \tilde{m}\rangle^s}f_1(h)f_2(\tilde{m})g(k_2)\\
        &\leq \sum_{k_2=h_2+m_2}\frac{1}{|h_2|^{2\alpha}}g(k_2)\sum_{h_1}f_1(h_1,h_2)f_2(-h_1,m_2)\\
        &\leq \sum_{k_2=h_2+m_2}\frac{g(k_2)}{|h_2|^{2\alpha}}\left(\sum_{h_1}|f_1(h_1,h_2)|^2\right)^{\frac{1}{2}}\left(\sum_{h_1}|f_2(-h_1,m_2)|^2\right)^{\frac{1}{2}}\\
        &= \sum_{h_2}\sum_{k_2}\frac{g(k_2)}{|h_2|^{2\alpha}}\left(\sum_{h_1}|f_1(h_1,h_2)|^2\right)^{\frac{1}{2}}\left(\sum_{h_1}|f_2(-h_1,k_2-h_2)|^2\right)^{\frac{1}{2}}\\
        &\leq \norm{g}_{l_2}\norm{f_2}_{l_2}\sum_{h_2}\frac{1}{|h_2|^{2\alpha}}\left(\sum_{h_1}|f_1(h_1,h_2)|^2\right)^{\frac{1}{2}}\\
        &\leq \norm{g}_{l_2}\norm{f_1}_{l_2}\norm{f_2}_{l_2}\left(\sum_{h_2}\frac{1}{|h_2|^{4\alpha}}\right)^{\frac{1}{2}}.
    \end{align*}
    The last term is summable if $4\alpha>1$, i.e.
    \begin{equation}
        \alpha>\frac{1}{4}.
    \end{equation} 
    \item [2)] $|h_2|\ll|m_2|\Rightarrow|m_2|\sim|k_2|$.
    \begin{align*}
        \frac{|k_2|^{s-2\alpha+1}}{|h|^{s+1}|\tilde{m}|^{s+1}}&\sim\frac{|m_2|^{s-2\alpha+1}}{(|h_1|^2+|h_2|^2)^{\frac{s+1}{2}}(|h_1|^2+|m_2|^2)^{\frac{s+1}{2}}}\\
        &\leq\frac{|m_2|^{s-2\alpha+1}}{|m_2|^{s+1}}\leq \frac{1}{|m_2|^{2\alpha}}
    \end{align*}
    because $|h|^{s+1}=(|h_1|^2+|h_2|^2)^{\frac{s+1}{2}}\geq 1$ and $|m_2|\neq0$ since $|m_2|\gg|h_2|$. Then,
     \begin{align*}
        &\sum_{k_2}\sum_{\substack{h_1 \\k_2=h_2+m_2}}\frac{V(k_2,h,\tilde{m})}{\langle h\rangle^s\langle \tilde{m}\rangle^s}f_1(h)f_2(\tilde{m})g(k_2)\\
        &\leq \sum_{k_2=h_2+m_2}\frac{1}{|m_2|^{2\alpha}}g(k_2)\sum_{h_1}f_1(h_1,h_2)f_2(-h_1,m_2)\\
       &\leq \sum_{k_2=h_2+m_2}\frac{g(k_2)}{|m_2|^{2\alpha}}\left(\sum_{h_1}|f_1(h_1,h_2)|^2\right)^{\frac{1}{2}}\left(\sum_{h_1}|f_2(-h_1,m_2)|^2\right)^{\frac{1}{2}}\\
        &= \sum_{m_2}\sum_{k_2}\frac{g(k_2)}{|m_2|^{2\alpha}}\left(\sum_{h_1}|f_1(h_1,k_2-m_2)|^2\right)^{\frac{1}{2}}\left(\sum_{h_1}|f_2(-h_1,m_2)|^2\right)^{\frac{1}{2}}\\
        &\leq \norm{g}_{l_2}\norm{f_1}_{l_2}\sum_{m_2}\frac{1}{|m_2|^{2\alpha}}\left(\sum_{h_1}|f_2(-h_1,m_2)|^2\right)^{\frac{1}{2}}\\
        &\leq \norm{g}_{l_2}\norm{f_1}_{l_2}\norm{f_2}_{l_2}\left(\sum_{m_2}\frac{1}{|m_2|^{4\alpha}}\right)^{\frac{1}{2}}.
    \end{align*}
    As above, the last term is summable if
    \begin{equation*}
        \alpha>\frac{1}{4}.
    \end{equation*}
    
    \item[3)] $|h_2|\sim|m_2|\Rightarrow|k_2|\sim|h_2|\sim|m_2|$.
    \begin{align*}
        \frac{|k_2|^{s-2\alpha+1}}{|h|^{s+1}|\tilde{m}|^{s+1}}&\sim\frac{|h_2|^{s-2\alpha+1}}{(|h_1|^2+|h_2|^2)^{\frac{s+1}{2}}(|h_1|^2+|m_2|^2)^{\frac{s+1}{2}}}
    \end{align*}
    Since $|k_2|\neq0$ at least one between $|h_2|$ and $|m_2|$ is not zero. Moreover $|k_2|\sim|h_2|\sim|m_2|$, then 
    \begin{align*}
        \frac{|k_2|^{s-2\alpha+1}}{|h|^{s+1}|\tilde{m}|^{s+1}}\lesssim \frac{1}{|h_2|^{2\alpha}}
    \end{align*}
    and we conclude as in the previous case.
    \end{itemize}
By $1),2),3)$ we have that 
\begin{align*}
    s&\geq0;\\
    \alpha&>\frac{1}{4}.
\end{align*}
Then, by \eqref{uniformbound1} and definitions of $f_1$ and $f_2$,
\begin{align*}
    \norm{\int_0^t e^{(t-t')\nu\partial_{yy}}G(\gamma) dt'}_{H_{y}^s} &\lesssim \norm{\gamma}^2_{H^s_{xy}}\int_0^t\frac{1}{(t-t')^\alpha}dt'\\
    &\sim  t^{1-\alpha}\norm{\gamma}^2_{H^s_{xy}},
\end{align*}
for $\frac{1}{4}<\alpha<1$ and $s\geq0$. Finally 
\begin{align*}
    \norm{\int_0^t e^{(t-t')\nu\partial_{yy}}G(\gamma) dt'}_{L^\infty_\delta H_{y}^s} \lesssim  \delta^{1-\alpha}\norm{\gamma}^2_{L^\infty_\delta H^s_{xy}}
\end{align*}
for $\frac{1}{4}<\alpha<1$ and $s\geq0$.
\end{proof}

We now define the following functional,
\begin{equation}\label{phi1}
        \Phi_1(\mu,\gamma)=e^{\nu t\partial_{yy}}\mu_0+\int_0^t e^{\nu(t-t')\partial_{yy}}G(\gamma) dt',
\end{equation}
and as a consequence of the previous proposition we have, 
\begin{cor}\label{corollariou}
Let $t\in [0,\delta]$, $\delta\in \mathbb{R}_{+}$, $s\geq 0$, $\alpha\in(\frac{1}{2},1)$. Then,
\begin{equation*}
    \norm{\Phi_1(\mu,\gamma)}_{L^\infty_\delta H^s_{y}}\lesssim  \norm{\mu_0}_{H^s_{y}}+\delta^{1-\alpha}\norm{\gamma}^2_{L^\infty_\delta H^s_{xy}}.
\end{equation*}
\end{cor}

\begin{proof}
\begin{equation*}
    \norm{\Phi_1(\mu,\gamma)}_{L^\infty_\delta H^s_{y}}\leq \underbrace{\norm{e^{\nu t\partial_{yy}}\mu_0}_{L^\infty_\delta H^s_{y}}}_{A}+\underbrace{ \norm{\int_0^t e^{(t-t')\nu\partial_{yy}}G(\gamma) dt'}^2_{L^\infty_t H_{y}^s}}_{B}.
\end{equation*}
By Proposition \ref{prop1u} we have
\begin{equation*}
    B\lesssim \delta^{1-\alpha}\norm{\gamma}^2_{L^\infty_\delta H^s_{xy}}.
\end{equation*}
On the other hand,
\begin{align*}
    A=\norm{e^{\nu t\partial_{yy}}\mu_0}_{L^\infty_\delta H^s_{y}}&=\sup_{t\in[0,\delta]}\left(\sum_{k_2}\left|e^{-t\nu k^2_2}\widehat{\mu}_0(k_2)\right|^2\langle k_2\rangle^{2s}\right)^{\frac{1}{2}}\\
    &\lesssim\left(\sum_{k_2}\left|\widehat{\mu}_0(k_2)\right|^2\langle k_2\rangle^{2s}\right)^{\frac{1}{2}}=\norm{\mu_0}_{H^s_{y}}.
\end{align*}
\end{proof}

\subsubsection{Eddy-Vorticity Equation Estimate}
Similarly, we study the equation for $\gamma$.

\begin{prop}\label{bound}
Let $t\in [0,\delta]$, $\delta\in \mathbb{R}_{+}$, $s\geq0$, $\alpha\in(\frac{3}{4},1)$. Then,
\begin{align*}
    \norm{\int_0^t e^{\tilde{D}^2(t-t')}F(\mu,\gamma)\ dt'}_{L^\infty_\delta H_{xy}^s}&\\
    \lesssim\delta^{1-\alpha}\left(\norm{\gamma}^2_{L^\infty_\delta H^s_{xy}}\right.&\left.+\norm{
    \mu}_{L^\infty_\delta H^s_{y}}\norm{\gamma}_{L^\infty_\delta H^s_{xy}}+\norm{\gamma}_{L^\infty_\delta H^s_{xy}}\right).
\end{align*}
\end{prop}

\begin{proof}
By definition,
\begin{align*}
    \norm{\int_0^t e^{\tilde{D}^2(t-t')}F(\mu,\gamma)\ dt'}_{L^\infty_\delta H_{xy}^s}& \\
    = \sup_{t\in[0,\delta]}\left(\sum_k \left|\int_0^t\right.\right. &\left.\left.e^{(i\frac{k_1}{|k|^2}-\nu|k|^2)(t-t')}\widehat{F(\mu,\gamma)}(k)dt'\right|^2\langle k\rangle^{2s}\right)^{\frac{1}{2}}
\end{align*}
where, 
\begin{align*}
    \widehat{F(\mu,\gamma)}&(k)\\
    &=\mathfrak{F}\left((\partial_y\nabla^{-2}\zeta')\zeta'_x-(\partial_x\nabla^{-2}\zeta')\zeta'_y-\mu\gamma_x+c_0\gamma_x+\partial_y\overline{[(\partial_x\nabla^{-2}\zeta')\zeta']}\right)\\
    &=-\left(i\frac{k_2}{|k|^2}\widehat{\gamma}\right)*(ik_1\widehat{\gamma})+\left(i\frac{k_1}{|k|^2}\widehat{\gamma}\right)*\left(ik_2\widehat{\gamma}\right)-\widehat{\mu}(k_2)*ik_1\widehat{\gamma}(k)\\
    &\left.+c_0ik_1\widehat{\gamma}(k)-ik_2\left(i\frac{k_1}{|k|^2}\widehat{\gamma}*\widehat{\gamma}\right)\right|_{[0,k_2]}\\
    &=\sum_{k=h+m}\left(\frac{h_2m_1}{|h|^2}-\frac{h_1m_2}{|h|^2}\right)\widehat{\gamma}(h)\widehat{\gamma}(m)-\sum_{\substack{k_1=h_1+0 \\k_2=h_2+m_2}} ih_1\widehat{\mu}(m_2)\widehat{\gamma}(h)\\
    &+c_0ik_1\widehat{\gamma}(k)+\sum_{\substack{0=k_1=h_1+m_1 \\k_2=h_2+m_2}} \frac{k_2h_1}{|h|^2}\widehat{\gamma}(h)\widehat{\gamma}(m).
\end{align*}
For now, we only consider 
\begin{align*}
      \bigg\|\int_0^t e^{\tilde{D}^2(t-t')}&F(\mu,\gamma)\ dt'\bigg\|^2_{H_{xy}^s}\\
      \leq &\int_0^t \norm{\int_0^t e^{\tilde{D}^2(t-t')}F(\mu,\gamma)\ dt'}^2_{H_{xy}^s} dt'\\
      =&\int_0^t \sum_k\left|  e^{(i\frac{k_1}{|k|^2}-\nu|k|^2)(t-t')}\widehat{F(\mu,\gamma)}(k)\right|^2\langle k\rangle^{2s}dt'\\
      =&\int_0^t \sum_k\left|  e^{(i\frac{k_1}{|k|^2}-\nu|k|^2)(t-t')}\bigg(\sum_{k=h+m}\left(\frac{h_2m_1}{|h|^2}-\frac{h_1m_2}{|h|^2}\right)\widehat{\gamma}(h)\widehat{\gamma}(m)\right.\\
      &+\sum_{\substack{k_1=h_1 \\k_2=h_2+m_2}} ih_1\widehat{\mu}(m_2)\widehat{\gamma}(h)
      +c_0ik_1\widehat{\gamma}(k)\\
      &+\sum_{\substack{k_1=h_1+m_1=0 \\k_2=h_2+m_2}} \left.\frac{k_2h_1}{|h|^2}\widehat{\gamma}(h)\widehat{\gamma}(m)\bigg)\right|^2\langle k\rangle^{2s}dt'
\end{align*}
\begin{align*}
      \leq\int_0^t \sum_k\Bigg[  e^{-\nu|k|^2(t-t')}\langle k\rangle^{s}&\bigg(\sum_{k=h+m}\left(\frac{|h||m|+|h||m|}{|h|^2}\right)|\widehat{\gamma}(h)||\widehat{\gamma}(m)|\\
      &+\sum_{\substack{k_1=h_1 \\k_2=h_2+m_2}} |h||\widehat{\mu}(m_2)||\widehat{\gamma}(h)|
      +c_0|k||\widehat{\gamma}(k)|\\
      &+\sum_{\substack{k_1=h_1+m_1=0 \\k_2=h_2+m_2}} \frac{|k||h|}{|h|^2}|\widehat{\gamma}(h)||\widehat{\gamma}(m)|\bigg)\Bigg]^2 dt'.
\end{align*}
Multiplying and dividing by $(\nu|k|^2(t-t'))^\alpha$, 
\begin{align*}
     =\int_0^t \sum_k\Bigg[  \frac{e^{-\nu|k|^2(t-t')}(\nu|k|^2(t-t'))^\alpha}{(\nu|k|^2(t-t'))^\alpha}&\langle k\rangle^{s}\Bigg(\sum_{k=h+m}2\frac{|m|}{|h|}|\widehat{\gamma}(h)||\widehat{\gamma}(m)|\\
     +&\sum_{\substack{k_1=h_1 \\k_2=h_2+m_2}}|h||\widehat{\mu}(m_2)||\widehat{\gamma}(h)|
     +c_0|k||\widehat{\gamma}(k)|\\
     +&\sum_{\substack{k_1=h_1+m_1=0 \\k_2=h_2+m_2}}\frac{|k|}{|h|}|\widehat{\gamma}(h)||\widehat{\gamma}(m)|\Bigg)\Bigg]^2 dt'.
\end{align*}
Since
\begin{equation}\label{uniformbound2}
    e^{-\nu|k|^2(t-t')}(\nu|k|^2(t-t'))^\alpha\leq C, \ \ \ C\in \mathbb{R}
\end{equation}
is uniformly bounded, we study 
\begin{align*}
     \sum_k\left[\frac{\langle k\rangle^{s}}{|k|^{2\alpha}}\right.\bigg(\sum_{k=h+m}&2\frac{|m|}{|h|}|\widehat{\gamma}(h)||\widehat{\gamma}(m)|+\sum_{\substack{k_1=h_1+0 \\k_2=h_2+m_2}} |h||\widehat{\mu}(m_2)||\widehat{\gamma}(h)|\\
     &
     +c_0|k||\widehat{\gamma}(k)|+\sum_{\substack{h_1+m_1=0 \\k_2=h_2+m_2}} \frac{|k|}{|h|}|\widehat{\gamma}(h)||\widehat{\gamma}(m)|\bigg)\bigg]^2.
\end{align*}
By duality we want to show
\begin{align}\label{versionesup1}
    \sup_{\norm{g}_{l^2}\leq 1}\sum_k\frac{\langle k\rangle^{s}}{|k|^{2\alpha}}\Bigg(\sum_{k=h+m}&2\frac{|m|}{|h|}|\widehat{\gamma}(h)||\widehat{\gamma}(m)|+\sum_{\substack{k_1=h_1+0 \\k_2=h_2+m_2}} |h||\widehat{\mu}(m_2)||\widehat{\gamma}(h)|\notag\\
     &
     +c_0|k||\widehat{\gamma}(k)|+\sum_{\substack{h_1+m_1=0 \\k_2=h_2+m_2}} \frac{|k|}{|h|}|\widehat{\gamma}(h)||\widehat{\gamma}(m)|\Bigg)g(k)\notag\\
    \lesssim\big(\norm{f_1}^2_{l^2}&\norm{g}_{l^2}+\norm{f_1}_{l^2}\norm{g}_{l^2}+\norm{f_1}_{l^2}\norm{f_2}_{l^2}\norm{g}_{l^2}\big)
\end{align}
where
\begin{align*}
    f_1(k)&:=\langle k\rangle^s|\widehat{\gamma}(k)|,\\
    f_2(k_2)&:=\langle k_2\rangle^s|\widehat{\mu}(k_2)|.
\end{align*}
We observe that if $\gamma,\ \mu\in H^s$ then $f_i\in l^2$ for $i=1,2$.\\
Since
\begin{equation*}
    \sup(a+b)\leq\sup(a)+\sup(b),
\end{equation*}
we split \eqref{versionesup1} and study the following problems 
\begin{align}
    &\sup_{\norm{g}_{l^2}\leq 1}\sum_k\frac{\langle k\rangle^{s}}{|k|^{2\alpha}}\sum_{k=h+m}\frac{|m|}{|h|}\frac{1}{\langle h\rangle^s\langle m\rangle^s}f_1(h)f_1(m)g(k)\lesssim\norm{f_1}^2_{l^2}\norm{g}_{l^2} \label{in1}\\
    &\sup_{\norm{g}_{l^2}\leq 1}\sum_{k}\frac{\langle k\rangle^s}{|k|^{2\alpha}}\sum_{\substack{h_1+m_1=0 \\ k_2=h_2+m_2}}\frac{|k|}{|h|}\frac{1}{\langle h\rangle^s\langle m\rangle^s}f_1(h)f_1(m)g(k)\lesssim\norm{f_1}_{l^2}^2\norm{g}_{l^2}\label{in2}\\
    &\sup_{\norm{g}_{l^2}\leq 1}\sum_{k}c_0\frac{|k|}{|k|^{2\alpha}}f_1(k)g(k)\lesssim\norm{f_1}_{l^2}\norm{g}_{l^2}\label{in3}\\
    &\sup_{\norm{g}_{l^2}\leq 1}\sum_{k}\frac{\langle k\rangle^s}{|k|^{2\alpha}}\sum_{\substack{k_1=h_1 \\ k_2=h_2+m_2}}\frac{|h|}{\langle h\rangle^s\langle m_2\rangle^s}f_1(h)f_2(m_2)g(k)\lesssim\norm{f_1}_{l^2}\norm{f_2}_{l^2}\norm{g}_{l^2}.\label{in4}
\end{align}
We start with \eqref{in1} and since
\begin{equation}
    \langle k\rangle^s\sim|k|^s,
\end{equation}
we have
\begin{equation*}
    \frac{\langle k\rangle^{s}}{|k|^{2\alpha}}\frac{|m|}{|h|}\frac{1}{\langle h\rangle^s\langle m\rangle^s}\sim \frac{|k|^{s-2\alpha}}{|h|^{s+1}|m|^{s-1}}.
\end{equation*}
\begin{itemize}
    \item [1)] $|h|\gg|m| \Rightarrow |k|\sim|h|$.
    \begin{equation*}
        \frac{|k|^{s-2\alpha}}{|h|^{s+1}|m|^{s-1}}\sim\frac{1}{|h|^{2\alpha+1}|m|^{s-1}}\leq \frac{1}{|m|^{2\alpha+s}}.
    \end{equation*}
    Proceeding as in Proposition \ref{prop1u}, we get 
    \begin{align*}
        \sum_k\sum_{k=h+m} &\frac{1}{|m|^{2\alpha+s}} f_1(h)f_1(m)g(k)\\
        =&\sum_m\frac{1}{|m|^{2\alpha+s}}f_1(m)\sum_k f_1(k-m)g(k) \\
        \leq& \sum_m\frac{1}{|m|^{2\alpha+s}}f_1(m)\left(\sum_k f_1^2(k-m)\right)^{\frac{1}{2}}\left(\sum_k g^2(k)\right)^{\frac{1}{2}} \\
        \leq&\left(\sum_m\frac{1}{|m|^{4\alpha+2s}}\right)^{\frac{1}{2}}\left(\sum_m f_1^2(m)\right)^{\frac{1}{2}}\norm{f_1}_{l^2}\norm{g}_{l^2} \\
        \leq&\left( \sum_m \frac{1}{|m|^{4\alpha+2s}}\right)^{\frac{1}{2}}\norm{f_1}^2_{l^2}\norm{g}_{l^2}.
    \end{align*}
    The first term is summable if
    \begin{equation*}
        4\alpha+2s>2.
    \end{equation*}
    The worst case is when $s=0$, but in this situation it is sufficient to choose 
    \begin{equation*}
        \alpha>\frac{1}{2}.
    \end{equation*}
    \item[2)]$|h|\ll|m|\Rightarrow |k|\sim|m|$.
    \begin{equation*}
        \frac{|k|^{s-2\alpha}}{|h|^{s+1}|m|^{s-1}}\sim\frac{1}{|h|^{s+1}|m|^{2\alpha-1}}\overbrace{\leq}^{\alpha>\frac{1}{2}}\frac{1}{|h|^{2\alpha+s}}\overbrace{\leq}^{s\geq 0}\frac{1}{|h|^{2\alpha}}.
    \end{equation*}
    Then,
    \begin{align*}
         \sum_k\sum_{k=h+m}& \frac{1}{|h|^{2\alpha}}f_1(h)f_1(m)g(k)\\
        &=\sum_h\frac{1}{|h|^{2\alpha}}f_1(h)\sum_m f_1(m)g(h+m)\\
        &\leq\left( \sum_h \frac{1}{|h|^{4\alpha}}\right)^{\frac{1}{2}}\norm{f_1}^2_{l^2}\norm{g}_{l^2}
    \end{align*}
    which is summable if $\alpha>\frac{1}{2}$.
    
    \item[3)] $|h|\sim|m|\Rightarrow |k|\sim|h| \sim |m|$.
    \begin{equation*}
        \frac{|k|^{s-2\alpha}}{|h|^{s+1}|m|^{s-1}}\sim\frac{1}{|h|^{2\alpha+s}}\leq \frac{1}{|h|^{2\alpha}}.
    \end{equation*}
    we conclude as in the previous case.
\end{itemize}
By $1),2),3)$ we get
\begin{align*}
    s&\geq 0;\\
    \alpha&>\frac{1}{2}.
\end{align*}
Similarly, we study \eqref{in2}. We have  
\begin{equation*}
    \frac{|k|}{|h|\langle h\rangle^s\langle m\rangle^s}\sim \frac{|k|}{|h|^{s+1}|m|^s},
\end{equation*}
so that we write
\begin{equation}\label{equazione in k}
    \sum_{k}\frac{\langle k\rangle^s}{|k|^{2\alpha}}g(k)\sum_{\substack{k_1=0 \\ k_2=h_2+m_2}}\frac{|k|}{|h|^{s+1}|m|^s} f_1(h)f_1(m)
\end{equation}
If we define $\tilde{m}=(-h_1,m_2)$, equation \eqref{equazione in k} becomes 
\begin{equation*}
    \sum_{k}\frac{\langle k\rangle^s}{|k|^{2\alpha}}g(k)\sum_{ k=h+\tilde{m}}\frac{|k|}{|h|^{s+1}|\tilde{m}|^s} f_1(h)f_1(\tilde{m})
\end{equation*}

\begin{itemize}
    \item [1)]$|h|\gg|\tilde{m}|\Rightarrow|k|\sim|h|$.
    \begin{equation*}
        \frac{\langle k\rangle^s}{|k|^{2\alpha}}\frac{|k|}{|h|^{s+1}|\tilde{m}|^s}\sim \frac{1}{|h|^{2\alpha}|\tilde{m}|^s}\leq\frac{1}{|\tilde{m}|^{2\alpha}|\tilde{m}|^s}\leq \frac{1}{|\tilde{m}|^{2\alpha}}.
    \end{equation*}
    Then,
\begin{align*}
        \sum_{k}\sum_{ k=h+\tilde{m}}&\frac{1}{|\tilde{m}|^{2\alpha}} f_1(h)f_1(\tilde{m})g(k)\\
        =&\sum_{\tilde{m}}\frac{1}{|\tilde{m}|^{2\alpha}}f_1(\tilde{m})\sum_k f_1(k-\tilde{m})g(k) \\
        \leq& \sum_{\tilde{m}}\frac{1}{|\tilde{m}|^{2\alpha}}f_1(\tilde{m})\left(\sum_k f_1^2(k-\tilde{m})\right)^{\frac{1}{2}}\left(\sum_k g^2(k)\right)^{\frac{1}{2}} \\
        \leq&\left(\sum_{\tilde{m}}\frac{1}{|\tilde{m}|^{4\alpha}}\right)^{\frac{1}{2}}\left(\sum_{\tilde{m}} f_1^2(\tilde{m})\right)^{\frac{1}{2}}\norm{f_1}_{l^2}\norm{g}_{l^2} \\
        \leq&\left( \sum_{\tilde{m}} \frac{1}{|\tilde{m}|^{4\alpha}}\right)^{\frac{1}{2}}\norm{f_1}^2_{l^2}\norm{g}_{l^2}.
    \end{align*}
    The first term is summable if
    \begin{equation*}
        \alpha>\frac{1}{2}.
    \end{equation*}
\item[2)]$|h|\ll|\tilde{m}|\Rightarrow |k|\sim|\tilde{m}|$.
     \begin{equation*}
        \frac{\langle k\rangle^s}{|k|^{2\alpha}}\frac{|k|}{|h|^{s+1}|\tilde{m}|^s}\sim\frac{1}{|k|^{2\alpha-1}|h|^{s+1}}\overbrace{\leq}^{\alpha>\frac{1}{2}} \frac{1}{|h|^{2\alpha-1}|h|^{s+1}}\leq \frac{1}{|h|^{2\alpha}}.
    \end{equation*}
    Then,
    \begin{align*}
         \sum_k\sum_{k=h+\tilde{m}}& \frac{1}{|h|^{2\alpha}}f_1(h)f_1(\tilde{m})g(k)\\
        &=\sum_h\frac{1}{|h|^{2\alpha}}f_1(h)\sum_{\tilde{m}} f_1(\tilde{m})g(h+\tilde{m})\\
        &\leq\left( \sum_h \frac{1}{|h|^{4\alpha}}\right)^{\frac{1}{2}}\norm{f_1}^2_{l^2}\norm{g}_{l^2},
    \end{align*}
    which is summable if $\alpha>\frac{1}{2}$.
    
    \item[3)] $|h|\sim|\tilde{m}|\Rightarrow |k|\sim|h| \sim |\tilde{m}|$.
    \begin{equation*}
        \frac{\langle k\rangle^s}{|k|^{2\alpha}}\frac{|k|}{|h|^{s+1}|\tilde{m}|^s}\lesssim\frac{1}{|h|^{2\alpha}}.
    \end{equation*}
    We conclude as in the previous case.
\end{itemize}
By $1),2),3)$ we get
\begin{align*}
    s&\geq0 ;\\
    \alpha&>\frac{1}{2}.
\end{align*}
For \eqref{in3}, if we assume $\alpha\geq \frac{1}{2}$
\begin{align*}
    \sum_{k}c_0\frac{1}{|k|^{2\alpha-1}}f_1(k)g(k)&\leq c_0\left(\sum_k f_1^2(k)\right)^{\frac{1}{2}}\left(\sum_k g^2(k)\right)^{\frac{1}{2}}\\
    &\sim\norm{f_1}_{l^2}\norm{g}_{l_2},
\end{align*}
where we used $|k|\geq1$.

Finally, we study \eqref{in4}. We have
\begin{align*}
    &\sum_{k}\frac{\langle k\rangle^s}{|k|^{2\alpha}}g(k)\sum_{\substack{k_1=h_1 \\ k_2=h_2+m_2}}\frac{|h|}{\langle h\rangle^s\langle m_2\rangle^s}f_1(h)f_2(m_2)\\
    \sim&\sum_{k}\frac{| k|^s}{|k|^{2\alpha}}g(k)\sum_{\substack{k_1=h_1 \\ k_2=h_2+m_2}}\frac{|h|}{| h|^s |m_2|^s}f_1(h)f_2(m_2).
\end{align*}
\begin{itemize}
    \item [1)] $|h|\gg|m|=|m_2|\Rightarrow|k|\sim|h|$. Then,
    \begin{equation*}
        \frac{|k|^s}{|k|^{2\alpha}}\frac{|h|}{|h|^s|m_2|^s}\sim \frac{1}{|h|^{2\alpha-1}|m_2|^s}.
    \end{equation*}
    If $\alpha>1/2$,
    \begin{equation*}
        \leq\frac{1}{|m_2|^{s-1+2\alpha}},
    \end{equation*}
    we get,
    \begin{align*}
        &\sum_{k}\sum_{\substack{k_1=h_1 \\ k_2=h_2+m_2}}\frac{1}{|m_2|^{2\alpha-1+s}}f_1(h_1,h_2)f_2(m_2)g(k_1,k_2)\\
        &\leq \sum_{k_2=h_2+m_2}\frac{1}{|m_2|^{2\alpha-1+s}}f_2(m_2)\overbrace{\left(\sum_{k_1}|f_1(k_1,h_2)|^2\right)^{\frac{1}{2}}}^{w(h_2)}\overbrace{\left(\sum_{k_1}|g(k_1,k_2)|^2\right)^{\frac{1}{2}}}^{v(k_2)}\\
        &=\sum_{k_2}\sum_{m_2}\frac{1}{|m_2|^{2\alpha-1+s}}f_2(m_2)w(k_2-m_2)v(k_2)\\
        &\leq \sum_{m_2}\frac{1}{|m_2|^{2\alpha-1+s}}f_2(m_2)\norm{f_1}_{l^2}\norm{g}_{l^2}\\
        &\leq\left(\sum_{m_2}\frac{1}{|m_2|^{4\alpha-2+2s}}\right)^{\frac{1}{2}}\norm{f_1}_{l^2}\norm{f_2}_{l^2}\norm{g}_{l^2}
    \end{align*}
    which is summable if $4\alpha-2+2s>1$. If $s=0$, then $\alpha>\frac{3}{4}$.
    \item[2)] $|m_2|\gg|h|$ $\Rightarrow$ $|k|\sim |m_2|$. Then,
    \begin{equation*}
        \frac{|k|^s}{|k|^{2\alpha}}\frac{|h|}{|h|^s|m_2|^s}\lesssim \frac{|m_2|}{|m_2|^{2\alpha}|h|^s}= \frac{1}{|m_2|^{2\alpha-1}|h|^s}\leq\frac{1}{|m_2|^{2\alpha-1}}
    \end{equation*}
    recalling that $|h|\geq1$ and $s\geq 0$. We conclude as in the previous case.
     
     \item[3)] $|m_2|\sim|h|$ $\Rightarrow$ $|k|\sim|m_2|\sim |h|$.Then,
    \begin{equation*}
        \frac{|k|^s}{|k|^{2\alpha}}\frac{|h|}{|h|^s|m_2|^s}\sim \frac{1}{|m_2|^{2\alpha-1+s}}
    \end{equation*}
     and we conclude as in the previous case.
\end{itemize}
Putting all the conditions on $\alpha$ and $s$ together, we get 
\begin{align*}
    s&\geq 0;\\
    1>\alpha&>\frac{3}{4}.
\end{align*}
By definition of $f_1$ and $f_2$ and using \eqref{uniformbound2}, we derive
\begin{align*}
    \bigg\|\int_0^t e^{\tilde{D}^2(t-t')}&F(\gamma,t') dt'\bigg\|_{ H_{xy}^s}\\
    &\lesssim\int_0^t\frac{1}{(t-t')^\alpha}dt'\left(\norm{\gamma}^2_{ H^s_{xy}}+\norm{\mu}_{ H^s_y}\norm{\gamma}_{ H^s_{xy}}+\norm{\gamma}_{ H^s_{xy}}\right)\\
    &\sim t^{1-\alpha}\left(\norm{\gamma}^2_{ H^s_{xy}}+\norm{\mu}_{ H^s_y}\norm{\gamma}_{ H^s_{xy}}+\norm{\gamma}_{ H^s_{xy}}\right).
\end{align*}
Finally,
\begin{align*}
    \bigg\|\int_0^t e^{\tilde{D}^2(t-t')}&F(\gamma,t') dt'\bigg\|_{L^\infty_\delta H_{xy}^s}\\
    &\lesssim\delta^{1-\alpha}\left(\norm{\gamma}^2_{L^\infty_\delta H^s_{xy}}+\norm{\mu}_{L^\infty_\delta H^s_y}\norm{\gamma}_{L^\infty_{\delta} H^s_{xy}}+\norm{\gamma}_{L^\infty_{\delta} H^s_{xy}}\right).
\end{align*}
for $s\geq0$ and $\frac{3}{4}<\alpha<1$.
\end{proof}

Let us introduce the following functional,
\begin{equation}\label{phi2}
    \Phi_2(\mu,\gamma)=e^{t\tilde{D}^2}\gamma_0+\int_0^t e^{\tilde{D}^2(t-t')}F(\mu,\gamma)dt'.
\end{equation}
An immediate consequence of the previous proposition is,

\begin{cor}\label{corupperbound}
Let $t\in [0,\delta]$, $\delta\in \mathbb{R}_{+}$, $s\geq 0$, $\alpha\in(\frac{3}{4},1)$. Then,
\begin{align*}
    \|\Phi_2(\mu,\gamma)&\|_{L^\infty_\delta H^s_{xy}}\\
    &\lesssim \norm{\gamma_0}_{H^s_{xy}}+\delta^{1-\alpha}\norm{\gamma}_{L^\infty_\delta H^s_{xy}}\left(\norm{\gamma}_{L^\infty_\delta H^s_{xy}}+\norm{\mu}_{L^\infty_\delta H^s_{y}}+1\right).
\end{align*}
\end{cor}
\begin{proof}
\begin{equation*}
    \norm{\Phi_2(\gamma)}_{L^\infty_\delta H^s_{xy}}\leq \underbrace{\norm{e^{t\tilde{D}^2}\gamma_0}_{L^\infty_\delta H^s_{xy}}}_{A}+\underbrace{ \norm{\int_0^t e^{\tilde{D}^2(t-t')}F(\gamma,t') dt'}_{L^\infty_\delta H_{xy}^s}}_{B}
\end{equation*}
By Proposition \ref{bound} we have
\begin{equation*}
    B\lesssim\delta^{1-\alpha}\left(\norm{\gamma}^2_{L^\infty_\delta H^s_{xy}}+\norm{\mu}_{L^\infty_\delta H^s_{y}}\norm{\gamma}_{L^\infty_\delta H^s_{xy}}+\norm{\gamma}_{L^\infty_\delta H^s_{xy}}\right).
\end{equation*}
On the other hand, 
\begin{align*}
    A=\norm{\mathfrak{F}(e^{t\tilde{D}^2}\gamma_0)}_{H^s}&=\left(\sum_k\left|e^{(i\frac{k_1}{|k|^2}-\nu|k|^2)}\widehat{\gamma}_0(k)\right|^2\langle k\rangle^{2s}\right)^\frac{1}{2}\\
    &\lesssim\left(\sum_k|\widehat{\gamma}_0(k)|^2\langle k\rangle^{2s}\right)^\frac{1}{2}\\
    &=\norm{\gamma_0}_{H^s_{xy}}.
\end{align*}
\end{proof}

\subsubsection{The $\Phi$-Functional}
With $\Phi_1$ defined in \eqref{phi1} and $\Phi_2$ defined in \eqref{phi2}, let us define 
\begin{align*}
    \Phi&(\mu,\gamma)=(\Phi_1,\Phi_2)\\
    &=\left(e^{\nu t\partial_{yy}}\mu_0+\int_0^t e^{\nu(t-t')\partial_{yy}}G(\gamma)\ dt',\ e^{\tilde{D}^2t}\gamma_0+\int_0^t e^{\tilde{D}^2(t-t')}F(\mu,\gamma)\ dt' \right)
\end{align*}
By Corollary \ref{corollariou} and \ref{corupperbound} we have
\begin{cor}
Let $t\in [0,\delta]$, $\delta\in \mathbb{R}_{+}$, $s\geq 0$, $\alpha\in(\frac{3}{4},1)$. Then,
\begin{align}\label{upperboundsystem}
    &\norm{\Phi(\mu,\gamma)}_{X^{s,\delta}}\leq C_1\norm{\mu_0}_{H^s_{y}}+ C_2\norm{\gamma_0}_{H^s_{xy}}\\
    &+C_3\delta^{1-\alpha}\norm{\gamma}^2_{L^\infty_\delta H^s_{xy}}+C_4\delta^{1-\alpha}\norm{\gamma}_{L^\infty_\delta H^s_{xy}}\left(\norm{\gamma}_{L^\infty_{\delta} H^s_{xy}}+\norm{\mu}_{L^\infty_\delta H^s_y}+1\right).\notag
\end{align}
\end{cor}

\begin{proof}
Since 
\begin{equation*}
    \norm{\Phi(\mu,\gamma)}_{X^{s,\delta}}=\norm{\Phi_1(\mu,\gamma)}_{L^\infty_t H^s_y}+\norm{\Phi_2(\mu,\gamma)}_{L^\infty_\delta H^s_{xy}},
\end{equation*}
we use the computations made in Corollary \ref{corollariou} and \ref{corupperbound} which give us 
\begin{align*}
    \norm{\Phi_1(\mu,\gamma)}_{L^\infty_t H^s_{xy}}&\leq C_1\norm{\mu_0}_{H^s_{y}}+C_3\delta^{1-\alpha}\norm{\gamma}^2_{L^\infty_\delta H^s_{xy}}\\
    \norm{\Phi_2(\mu,\gamma)}_{L^\infty_t H^s_{xy}}&\leq C_3\norm{\gamma_0}_{H^s_{xy}}\\
    &+C_4\delta^{1-\alpha}\norm{\gamma}_{L^\infty_\delta H^s_{xy}}\left(\norm{\gamma}_{L^\infty_\delta H^s_{xy}}+\norm{\mu}_{L^\infty_\delta H^s_{y}}+1\right).
\end{align*}
Combining the results we get \eqref{upperboundsystem}.
\end{proof}

\subsection{Local Well-Posedness}

Let us consider      
\begin{equation*}
    B(0,R)\subset{X^{s,\delta}}
\end{equation*}
where $R:=2\left(C_1\norm{\mu_0}_{H^s_{y}}+ C_2\norm{\gamma_0}_{H^s_{xy}}\right)$. If we set $C_1,C_2\gg 1$ such that $R\geq1$ and we fix $\delta$ sufficiently small such that 
\begin{align}\label{conditiondelta}
      C_i\delta^{1-\alpha}R<\frac{1}{8},\text{ with } i=3,4,
\end{align}
then by \eqref{upperboundsystem} we get 
\begin{equation*}
    \Phi:B(0,R)\to B(0,R).
\end{equation*}
We remark that \eqref{conditiondelta} implies
\begin{equation}\label{explicitdelta}
    \delta\sim R^{-\frac{1}{1-\alpha}}\sim(\norm{\gamma_0}_{H^s_{xy}}+\norm{\mu_0}_{H^s_y})^{-\frac{1}{1-\alpha}}.
\end{equation}

\subsubsection{Contraction}
We now show that $\Phi$ is a contraction on $B(0,R)$. As in the previous section, we proceed by doing one estimate at a time and then combine the results.

\begin{prop}\label{2u}
Let $(\mu_1, \gamma_1),\ (\mu_2, \gamma_2)\in 
B(0,R) 
$ be vectors functions such that $\mu_{1}(0,y)=\mu_{2}(0,y)=:\mu_0$ and let $s\geq 0$, $\alpha\in(\frac{1}{2},1)$ Then, 
\begin{equation*}
    \norm{\Phi_1(\mu_1,\gamma_1)-\Phi_1(\mu_2,\gamma_2)}_{L^\infty_{\delta}H^s_{y}}\lesssim \delta^{1-\alpha}R\norm{(\mu_1,\gamma_1)-(\mu_2,\gamma_2)}_{X^{s,\delta}}.
\end{equation*}
\end{prop}
\begin{proof}
By definition we have
\begin{align*}
    \Phi_1(\mu_1,\gamma_1)&=e^{\nu t\partial_{yy}}\mu_0+\int_0^t e^{\nu(t-t')\partial_{yy}}G(\gamma_1) dt',\\
    \Phi_1(\mu_2,\gamma_2)&=e^{\nu t\partial_{yy}}\mu_0+\int_0^t e^{\nu(t-t')\partial_{yy}}G(\gamma_2) dt'.
\end{align*}
Then,
\begin{equation*}
    \Phi_1(\mu_1,\gamma_1)-\Phi_1(\mu_2,\gamma_2)=\int_0^t e^{\nu(t-t')\partial_{yy}}\left(G(\gamma_1)-G(\gamma_2)\right) dt'
\end{equation*}
where,
\begin{align*}
    G(\gamma_1)-G(\gamma_2)&=\partial_y\left(\overline{[\partial_x\nabla^{-2}\gamma_1\partial_y\nabla^{-2}\gamma_1]-[\partial_x\nabla^{-2}\gamma_2\partial_y\nabla^{-2}\gamma_2]}\right)\\
    =&\partial_y\left(\overline{[\partial_x\nabla^{-2}(\gamma_1-\gamma_2)\partial_y\nabla^{-2}\gamma_1]+[\partial_x\nabla^{-2}\gamma_2\partial_y\nabla^{-2}(\gamma_1-\gamma_2)]}\right)
\end{align*}
and 
\begin{align*}
      \norm{\Phi_1(\mu_1,\gamma_1)-\Phi_1(\mu_2,\gamma_2)}_{L^\infty_\delta H_{y}^s}=\norm{\int_0^t e^{\nu(t-t')\partial_{yy}}\left(G(\gamma_1)-G(\gamma_2)\right) dt'}_{L^\infty_\delta H_{y}^s}.
\end{align*}
As seen in Proposition \ref{prop1u}, by duality we can study the following two equivalent quantities
\begin{align*}
    \sup_{\norm{g}_{l^2}\leq 1}&\sum_{k_2}\sum_{\substack{h_1 \\k_2=h_2+m_2}}\frac{V(k_2,h,\tilde{m})}{\langle h\rangle^s\langle \tilde{m}\rangle^s}f_1(h)f_2(\tilde{m})g(k_2)\lesssim\norm{f_1}_{l^2}\norm{f_2}_{l^2}\norm{g}_{l^2}\\
    \sup_{\norm{g}_{l^2}\leq 1}&\sum_{k_2}\sum_{\substack{h_1 \\k_2=h_2+m_2}}\frac{V(k_2,h,\tilde{m})}{\langle h\rangle^s\langle \tilde{m}\rangle^s}f_3(h)f_1(\tilde{m})g(k_2)\lesssim \norm{f_1}_{l^2}\norm{f_3}_{l^2}\norm{g}_{l^2}.
\end{align*}
where
\begin{align*}
    \tilde{m}&:=(-h_1,m_2);\\
    V(k_2,h,\tilde{m})&:=\frac{\langle k_2\rangle^s}{|k_2|^{2\alpha-1}|h||\tilde{m}|};\\
    f_1(k)&:=\langle k\rangle^s|(\widehat{\gamma_1-\gamma_2})(k)|;\\
    f_2(k)&:=\langle k\rangle^s|\widehat{\gamma_1}(k)|;\\
    f_3(k)&:=\langle k\rangle^s|\widehat{\gamma_2}(k)|.
\end{align*}
Since we have the same coefficients as in the Proposition \ref{prop1u}, using the same computations we get
\begin{align*}
      \|\Phi_1(\mu_1,\gamma_1)-\Phi_1(\mu_2,\gamma_2)&\|_{L^\infty_\delta H_{y}^s}\\
      &\lesssim \delta^{1-\alpha}\norm{\gamma_1-\gamma_2}_{L^\infty_\delta H^s_{xy}}\left(\norm{\gamma_1}_{L^\infty_\delta H^s_{xy}}+\norm{\gamma_2}_{L^\infty_\delta H^s_{xy}}\right)\\
      &\lesssim \delta^{1-\alpha}R\norm{\gamma_1-\gamma_2}_{L^\infty_\delta H^s_{xy}}\\
      &\lesssim \delta^{1-\alpha}R\norm{(\mu_1,\gamma_1)-(\mu_2,\gamma_2)}_{X^{s,\delta}}.
\end{align*}
for $s\geq0$ and $\frac{1}{2}<\alpha<1$.
\end{proof}

\begin{prop}\label{contraction}
Let $(\mu_1,\gamma_1),\ (\mu_2,\gamma_2)\in B(0,R)$ be vectors functions such that $\gamma_1(0,x,y)=\gamma_2(0,x,y)=:\gamma_0(x,y)$ and let $s\geq 0$, $\alpha\in(\frac{3}{4},1)$. Then, 
\begin{equation*}
    \norm{\Phi_2(\mu_1,\gamma_1)-\Phi_2(\mu_2,\gamma_2)}_{L^\infty_{\delta}H^s_{xy}}\lesssim(1+3R) \delta^{1-\alpha}\norm{(\mu_1,\gamma_1)- (\mu_2,\gamma_2)}_{X^{s,\delta}},
\end{equation*}
with $\alpha\in\left(\frac{3}{4},1\right)$.
\end{prop}

\begin{proof}
By definition we have 
\begin{align*}
    \Phi_2(\mu_1,\gamma_1)&=e^{t\tilde{D}^2}\gamma_0+\int_0^t e^{(t-t')\tilde{D}^2}F(\mu_1,\gamma_1)\ dt'\\
    \Phi_2(\mu_2,\gamma_2)&=e^{t\tilde{D}^2}\gamma_0+\int_0^t e^{(t-t')\tilde{D}^2}F(\mu_2,\gamma_2)\ dt',
\end{align*}
where
\begin{align*}
    F(\mu,\gamma)=\overbrace{\gamma_x\partial_y\nabla^{-2}\gamma}^{A}-\overbrace{\gamma_y\partial_x\nabla^{-2}\gamma}^{B}+\overbrace{\partial_y[\overline{(\gamma\partial_x\nabla^{-2}\gamma)}}^{C} ]-\overbrace{\mu\gamma_x}^{D}+\overbrace{c_0\gamma_x}^{E}. 
\end{align*}
Then,
\begin{align*}
    &\norm{\Phi_2(\mu_1,\gamma_1)-\Phi_2(\mu_2,\gamma_2)}_{L^\infty_{\delta}H^s_{xy}}\\
    &=\norm{\int_0^t e^{(i\frac{k_1}{|k|^2}-\nu|k|^2)(t-t')}(F(\mu_1,\gamma_1)-F(\mu_2,\gamma_2))\ dt'}_{L^\infty_{\delta}H^s_{x,y}}\\
    &=\sup_{t\in[0,\delta]} \left(\sum_k \left|\int_0^t e^{(i\frac{k_1}{|k|^2}-\nu|k|^2)(t-t')}(\widehat{F(\mu_1,\gamma_1)-F(\mu_2,\gamma_2)})\ dt'\right|^2\langle k\rangle^{2s}\right)^{\frac{1}{2}}.
\end{align*}
We observe that
\begin{align*}
    A_{\gamma_1}&={\gamma_1}_x\partial_y\nabla^{-2}\gamma_1; \ \ \ \ A_{\gamma_2}={\gamma_2}_x\partial_y\nabla^{-2}\gamma_2.\\
    A_{\gamma_1}-A_{\gamma_2}&={\gamma_1}_x\partial_y\nabla^{-2}\gamma_1-{\gamma_2}_x\partial_y\nabla^{-2}{\gamma_2}\\
    &=({\gamma_1}_x-{\gamma_2}_x)\partial_y\nabla^{-2}\gamma_1+{\gamma_2}_x\partial_y\nabla^{-2}(\gamma_1-{\gamma_2}).
\end{align*}
\begin{align*}
    B_{\gamma_1}&={\gamma_1}_y\partial_x\nabla^{-2}\gamma_1;\ \ \ \ B_{\gamma_2}={\gamma_2}_y\partial_x\nabla^{-2}{\gamma_2}.\\
    B_{\gamma_1}-B_{\gamma_2}&={\gamma_1}_y\partial_x\nabla^{-2}\gamma_1-{\gamma_2}_y\partial_x\nabla^{-2}{\gamma_2}\\
    &=({\gamma_1}_y-{\gamma_2}_y)\partial_x\nabla^{-2}\gamma_1+{\gamma_2}_y\partial_x\nabla^{-2}(\gamma_1-{\gamma_2}).
\end{align*}
\begin{align*}        C_{\gamma_1}&=\partial_y\overline{(\gamma_1\partial_x\nabla^{-2}\gamma_1)}; \ \ \ \ C_{\gamma_2}=\partial_y\overline{({\gamma_2}\partial_x\nabla^{-2}{\gamma_2})}.\\
    C_{\gamma_1}-C_{\gamma_2}&=\partial_y\overline{(\gamma_1\partial_x\nabla^{-2}\gamma_1)}-\partial_y\overline{({\gamma_2}\partial_x\nabla^{-2}{\gamma_2})}\\
    &=\partial_y\overline{((\gamma_1-{\gamma_2})\partial_x\nabla^{-2}\gamma_1)}+\partial_y\overline{({\gamma_2}\partial_x\nabla^{-2}(\gamma_1-{\gamma_2}))}
\end{align*}
\begin{align*}
    D_{\mu_1,\gamma_1}&=\mu_1{\gamma_1}_x; \ \ \ \ D_{\mu_2,\gamma_2}=\mu_2{\gamma_2}_x.\\
    D_{\mu_1,\gamma_1}-D_{\mu_2,\gamma_2}&=\mu_1{\gamma_1}_x-\mu_2{\gamma_2}_x\\
    &=\mu_1(\gamma_1-\gamma_2)_x+(\mu_1-\mu_2){\gamma_2}_x
\end{align*}
\begin{align*}
    E_{\gamma_1}&=c_0{\gamma_1}_x; \ \ \ \ E_{\gamma_2}=c_0{\gamma_2}_x.\\
    E_{\gamma_1}-E_{\gamma_2}&=c_0\left({\gamma_1}_x-{\gamma_2}_x\right).
\end{align*}
Then,
\begin{align*}
    F(\mu_1,\gamma_1)-F({\mu_2,\gamma_2})&=({\gamma_1}_x-{\gamma_2}_x)\partial_y\nabla^{-2}\gamma_1+{\gamma_2}_x\partial_y\nabla^{-2}(\gamma_1-{\gamma_2})\\
    &-({\gamma_1}_y-{\gamma_2}_y)\partial_x\nabla^{-2}\gamma_1-{\gamma_2}_y\partial_x\nabla^{-2}(\gamma_1-{\gamma_2})\\
    &+\partial_y\overline{((\gamma_1-{\gamma_2})\partial_x\nabla^{-2}\gamma_1)}+\partial_y\overline{({\gamma_2}\partial_x\nabla^{-2}(\gamma_1-{\gamma_2}))}\\
    &-\mu_1(\gamma_1-\gamma_2)_x-(\mu_1-\mu_2){\gamma_2}_x+c_0\left({\gamma_1}_x-{\gamma_2}_x\right),
\end{align*}
and
\begin{align*}
    &\left(\widehat{F(\mu_1,\gamma_1)-F({\mu_2,\gamma_2})}\right)(k)\\
    &=ik_1(\widehat{\gamma_1-{\gamma_2}})(k)*(-i\frac{k_2}{|k|^2}\widehat{\gamma_1})(k)+(ik_1\widehat{{\gamma_2}})(k)*(-i\frac{k_2}{|k|^2}(\widehat{\gamma_1-{\gamma_2}}))(k)\\
    &-ik_2(\widehat{\gamma_1-{\gamma_2}})(k)*(-i\frac{k_1}{|k|^2}\widehat{\gamma_1})(k)-(ik_2\widehat{{\gamma_2}})(k)*(-i\frac{k_1}{|k|^2}(\widehat{\gamma_1-{\gamma_2}}))(k)\\
    &+\left.ik_2\left((\widehat{\gamma_1-{\gamma_2}})(k)*(-i\frac{k_1}{|k|^2}\widehat{\gamma_1})(k)+\widehat{{\gamma_2}}(k)*(-i\frac{k_1}{|k|^2})(\widehat{\gamma_1-{\gamma_2}})(k)\right)\right|_{[0,k_2]}\\
    &-\widehat{\mu_1}(k_2)*ik_1(\widehat{\gamma_1-{\gamma_2}})(k)-(\widehat{\mu_1-\mu_2})(k_2)*ik_1\widehat{\gamma_2}(k)+c_0ik_1(\widehat{\gamma-{\gamma_2}})(k)
\end{align*}
\begin{align*}
    &=\sum_{k=h+m}\frac{m_1h_2-m_2h_1}{|h|^2}\big[(\widehat{\gamma_1-{\gamma_2}})(m)\widehat{\gamma_1}(h)+(\widehat{\gamma_1-{\gamma_2}})(h)\widehat{{\gamma_2}}(m)\big]\\
    &+\sum_{(0,k_2)}\frac{k_2h_1}{|h|^2}\big[(\widehat{\gamma_1-{\gamma_2}})(m)\widehat{\gamma_1}(h)+(\widehat{\gamma_1-{\gamma_2}})(h)\widehat{{\gamma_2}}(m)\big]\\
    &-\sum_{\substack{k_1=h_1+0 \\k_2=h_2+m_2}} ih_1\big[\widehat{\mu_1}(m_2)(\widehat{\gamma_1-{\gamma_2}})(h)+(\widehat{\mu_1-\mu_2})(m_2)\widehat{\gamma_2}(h)\big]\\
    &+c_0ik_1(\widehat{\gamma_1-{\gamma_2}})(k).
\end{align*}
Therefore,
\begin{align*}
    \|\Phi_2(\mu_1,\gamma_1)-\Phi_2(&\mu_2,\gamma_2)\|_{L^\infty_{\delta}H^s_{xy}}\\
    =\sup_{t\in[0,\delta]} \left(\sum_k\right. &\left|\int_0^t\right. e^{(i\frac{k_1}{|k|^2}-\nu|k|^2)(t-t')}\\
    \cdot\left(\sum_{k=h+m}\right.&\frac{m_1h_2-m_2h_1}{|h|^2}\big[(\widehat{\gamma_1-{\gamma_2}})(m)\widehat{\gamma_1}(h)+(\widehat{\gamma_1-{\gamma_2}})(h)\widehat{{\gamma_2}}(m)\big]\\
    +&\sum_{\substack{k_1=0 \\k_2=h_2+m_2}}\frac{k_2h_1}{|h|^2}\big[(\widehat{\gamma_1-{\gamma_2}})(m)\widehat{\gamma_1}(h)+(\widehat{\gamma_1-{\gamma_2}})(h)\widehat{{\gamma_2}}(m)\big]\\
    -&\sum_{\substack{k_1=h_1+0 \\k_2=h_2+m_2}} ih_1\big[\widehat{\mu_1}(m_2)(\widehat{\gamma_1-{\gamma_2}})(h)+(\widehat{\mu_1-\mu_2})(m_2)\widehat{\gamma_2}(h)\big]\\
    +&c_0ik_1(\widehat{\gamma_1-{\gamma_2}})(k)\Bigg) dt'\bigg|^2\langle k\rangle^{2s}\Bigg)^{\frac{1}{2}}.
\end{align*}
As in the proof of Proposition \ref{bound}, by duality it is sufficient to show 
\begin{align*}
    &\sup_{\norm{g}_{l^2}\leq 1}\sum_k 2\frac{\langle k\rangle^{s}}{|k|^{2\alpha}}\sum_{k=h+m}\frac{|m|}{|h|}\frac{1}{\langle h\rangle^s\langle m\rangle^s}f_1(h)f_2(m)g(k)\lesssim\norm{f_1}_{l^2}\norm{f_2}_{l^2}\norm{g}_{l^2}
\end{align*}
\begin{align*}
    &\sup_{\norm{g}_{l^2}\leq 1}\sum_k 2\frac{\langle k\rangle^{s}}{|k|^{2\alpha}}\sum_{k=h+m}\frac{|m|}{|h|}\frac{1}{\langle h\rangle^s\langle m\rangle^s}f_2(h)f_3(m)g(k)\lesssim\norm{f_2}_{l^2}\norm{f_3}_{l^2}\norm{g}_{l^2}
\end{align*}
\begin{align*}
    &\sup_{\norm{g}_{l^2}\leq 1}\sum_{k}\frac{\langle k\rangle^s}{|k|^{2\alpha}}\sum_{\substack{k_1=0 \\ k_2=h_2+m_2}}\frac{|k|}{|h|\langle h\rangle^s\langle m\rangle^s}f_1(h)f_2(m)g(k)\lesssim\norm{f_1}_{l^2}\norm{f_2}_{l^2}\norm{g}_{l^2}
\end{align*}
\begin{align*}
    &\sup_{\norm{g}_{l^2}\leq 1}\sum_{k}\frac{\langle k\rangle^s}{|k|^{2\alpha}}\sum_{\substack{k_1=0 \\ k_2=h_2+m_2}}\frac{|k|}{|h|\langle h\rangle^s\langle m\rangle^s}f_2(h)f_3(m)g(k)\lesssim\norm{f_2}_{l^2}\norm{f_3}_{l^2}\norm{g}_{l^2}
\end{align*}
\begin{align*}
    &\sup_{\norm{g}_{l^2}\leq 1}\sum_{k}\frac{\langle k\rangle^s}{|k|^{2\alpha}}\sum_{\substack{k_1=h_1 \\ k_2=h_2+m_2}}\frac{|h|}{\langle h\rangle^s\langle m_2\rangle^s}f_2(h)f_4(m_2)g(k)\lesssim\norm{f_2}_{l^2}\norm{f_4}_{l^2}\norm{g}_{l^2}
\end{align*}
\begin{align*}
    &\sup_{\norm{g}_{l^2}\leq 1}\sum_{k}\frac{\langle k\rangle^s}{|k|^{2\alpha}}\sum_{\substack{k_1=h_1 \\ k_2=h_2+m_2}}\frac{|h|}{\langle h\rangle^s\langle m_2\rangle^s}f_3(h)f_5(m_2)g(k)\lesssim\norm{f_3}_{l^2}\norm{f_5}_{l^2}\norm{g}_{l^2}\\
    &\sup_{\norm{g}_{l^2}\leq 1}\sum_{k}c_0\frac{|k|}{|k|^{2\alpha}}f_2(k)g(k)\lesssim\norm{f_2}_{l^2}\norm{g}_{l^2},
\end{align*}
where we define
\begin{align*}
    f_1(k):=&\langle k \rangle^s|\widehat{\gamma_1}(k)|\\
    f_2(k):=&\langle k \rangle^s|(\widehat{\gamma_1-{\gamma_2}})(k)|\\
    f_3(k):=&\langle k \rangle^s|\widehat{{\gamma_2}}(k)|\\
    f_4(k_2):=&\langle k_2 \rangle^s|\widehat{\mu_1}(k_2)|\\
    f_5(k_2):=&\langle k_2 \rangle^s|(\widehat{\mu_1-\mu_2})(k_2)|.
\end{align*}
Since the coefficients of the sum are exactly those of Proposition \ref{bound}, we have
\begin{align*}
    s&\geq0,\\
    \alpha&>\frac{3}{4}
\end{align*}
and
\begin{align*}
    \|\Phi_2(\mu_1,\gamma_1)-\Phi_2({\mu_2,\gamma_2})&\|_{H^s_{xy}}\lesssim\int_0^t\frac{1}{(t-t')^\alpha}\ dt'(2\norm{\gamma_1-{\gamma_2}}_{H^s_{xy}}\norm{\gamma_1}_{H^s_{xy}}\\
    &+2\norm{\gamma_1-{\gamma_2}}_{H^s_{xy}}\norm{{\gamma_2}}_{H^s_{xy}}+\norm{\gamma_1-{\gamma_2}}_{H^s_{xy}}\norm{\mu_1}_{H^s_{y}}\\
    &+\norm{{\gamma_2}}_{H^s_{xy}}\norm{\mu_1-\mu_2}_{H^s_{y}}+\norm{\gamma_1-{\gamma_2}}_{H^s_{xy}})\\
    \sim t^{1-\alpha}&\left(\norm{\gamma_1-{\gamma_2}}_{H^s_{xy}}(\norm{\gamma_1}_{H^s_{xy}}+\norm{{\gamma_2}}_{H^s_{xy}}+\norm{\mu_1}_{H^s_{xy}}+1)\right.\\
    &\left.+\norm{{\gamma_2}}_{H^s_{xy}}\norm{\mu_1-\mu_2}_{H^s_{y}}\right)\\
    \lesssim (1+&3R) t^{1-\alpha}\left(\norm{\gamma_1-{\gamma_2}}_{H^s_{xy}}+\norm{\mu_1-\mu_2}_{H^s_{y}}\right).
\end{align*}
Then,
\begin{equation*}
    \norm{\Phi_2(\mu_1,\gamma_1)-\Phi_2({\mu_2,\gamma_2})}_{L^\infty_{\delta}H^s_{x,y}}\lesssim (1+3R) \delta^{1-\alpha}\norm{(\mu_1,\gamma_1)-(\mu_2,{\gamma_2})}_{X^{s,\delta}}.
\end{equation*}
with $s\geq0$ and $\frac{3}{4}< \alpha < 1$.
\end{proof}

Finally, combining Proposition \ref{2u} and \ref{contraction} we get
\begin{prop} 
Let $(\mu_1,\gamma_1),\ (\mu_2,\gamma_2)\in B(0,R)\subset{X^{s,\delta}}$ be vector functions and let $s\geq0$ and $\alpha\in(\frac{3}{4},1)$. Then 
\begin{equation}\label{contractionsystem}
    \norm{\Phi(\mu_1,\gamma_1)-\Phi(\mu_2,\gamma_2)}_{{X^{s,\delta}}}\leq C (1+3R) \delta^{1-\alpha}\norm{(\mu_1,\gamma_1)- (\mu_2,\gamma_2)}_{X^{s,\delta}}.
\end{equation}
\end{prop}

\begin{proof}
Since
\begin{align*}
    \norm{\Phi(\mu_1,\gamma_1)-\Phi(\mu_2,\gamma_2)}_{{X^{s,\delta}}}&=\norm{\Phi_1(\mu_1,\gamma_1)-\Phi_1(\mu_2,\gamma_2)}_{L^\infty_{\delta} H^s_y}\\
    &+\norm{\Phi_2(\mu_1,\gamma_1)-\Phi_2(\mu_2,\gamma_2)}_{L^\infty_{\delta} H^s_{xy}},
\end{align*}
and recalling that from Propositions \ref{2u} and \ref{contraction} we have
\begin{align*}
     \norm{\Phi_1(\mu_1,\gamma_1)-\Phi_1(\mu_2,\gamma_2)}&_{L^\infty_{\delta}H^s_{y}}\leq C_1\delta^{1-\alpha}R\norm{(\mu_1,\gamma_1)-(\mu_2,\gamma_2)}_{X^{s,\delta}}\\
     \norm{\Phi_2(\mu_1,\gamma_1)-\Phi_2(\mu_2,\gamma_2)}&_{L^\infty_{\delta}H^s_{xy}}\leq C_2 (1+3R) \delta^{1-\alpha}\norm{(\mu_1,\gamma_1)- (\mu_2,\gamma_2)}_{X^{s,\delta}}
\end{align*}
we get \eqref{contractionsystem}.
\end{proof}
If we set $\delta$ sufficiently small such that
\begin{equation*}
    C(1+3R)\delta^{1-\alpha}\leq\frac{1}{2},
\end{equation*}
then $\Phi$ is a contraction on $B(0,R)$.

\subsubsection{Well-Posedness on $X^{s,\delta}$}

\begin{cor}\label{fixedpointsystem}
Under the previous assumption,
\begin{equation*}
    \exists !\ (\mu,\gamma)\in B(0,R)\subset X^{s,\delta} \ | \ (\mu,\gamma)=\Phi(\mu,\gamma). 
\end{equation*}
\end{cor}
\begin{proof}
By fixed point theorem.
\end{proof}

Next step is to extend the existence and uniqueness of the solution from $B(0,R)$ to the whole space $X^{s,\delta}$.
\begin{prop}\label{estensionesoluzionesystem}
Suppose there exists $(\tilde{\mu},\tilde{\gamma})\in X^{s,\delta}$, $s\geq0$, solution of \eqref{sistemariscritto} such that 
\begin{align*}
    (\tilde{\mu},\tilde{\gamma})=&\left(e^{\nu t\partial_{yy}}\mu_0+\int_0^t e^{\nu(t-t')\partial_{yy}}G(\tilde{\gamma})\ dt';\right.\\ &\left.e^{\tilde{D}^2t}\gamma_0+\int_0^t e^{\tilde{D}^2(t-t')}F(\tilde{\mu},\tilde{\gamma})\ dt' \right)\text{ on }{X^{s,\delta}}.
\end{align*}
Let $(\mu,\gamma)\in B(0,R)$ be the functions defined in Corollary \ref{fixedpointsystem}. Then,
\begin{equation*}
    (\tilde{\mu},\tilde{\gamma})=(\mu,\gamma) \text{ on } X^{s,\delta}.
\end{equation*}
\end{prop}

\begin{proof}
Fix $0<\varepsilon<\delta$. By \eqref{contractionsystem}
\begin{align*}
     \norm{(\mu,\gamma)-(\tilde{\mu},\tilde{\gamma})}_{X^{s,\varepsilon}}&=\norm{\Phi(\mu,\gamma)-\Phi(\tilde{\mu},\tilde{\gamma})}_{L^\infty_{\varepsilon}H^s_{xy}}\\
     &\leq C (1+3R) \delta^{1-\alpha}\norm{(\mu_1,\gamma_1)- (\mu_2,\gamma_2)}_{X^{s,\varepsilon}}.
\end{align*}
Set $\varepsilon>0$ sufficiently small such that
\begin{equation*}
    C (1+3R) \varepsilon^{1-\alpha}\leq \frac{1}{2}.
\end{equation*}
Then,
\begin{equation*}
    \norm{(\mu,\gamma)-(\tilde{\mu},\tilde{\gamma})}_{X^{s,\varepsilon}}\leq \frac{1}{2}\norm{(\mu,\gamma)-(\tilde{\mu},\tilde{\gamma})}_{X^{s,\varepsilon}}
\end{equation*}
which is possible if and only if $(\tilde{\mu},\tilde{\gamma})=(\mu,\gamma)$  on  $X^{s,\varepsilon}$.\\
In particular $(\tilde{\mu}(\varepsilon),\tilde{\gamma}(\varepsilon))=(\mu(\varepsilon),\gamma(\varepsilon))$, therefore we can repeat the same argument on interval $[\varepsilon, 2\varepsilon]$ until we cover $[0,\delta]$.
\end{proof}

\section{Global Well-Posedness for Eddy-Mean Vorticity System}

\subsection{Upper Bounds of $L^2$-Norms by Initial Data}
To provide global well-posedness of \eqref{sistemariscritto}, an upper bound for the $L^2$ norms of $\overline{u}$ and $\zeta'$ is required, as mentioned in Step 4) of the introduction. This upper bound can be obtained from the norm of the initial data, as follows.

Let us start by recalling the equation \eqref{vorticityeq} that governs the dynamics of the problem, which is given by
\begin{equation*}
    \partial_t \zeta+ J(\psi,\zeta+\beta y)=\nu\nabla^2\zeta,
\end{equation*}
where $J(\psi,\zeta+\beta y)=u(\zeta+\beta y)_x+v(\zeta+\beta y)_y$, $u=-\psi_y,\ v=\psi_x$ and $\zeta=\nabla^2\psi$.

\textbf{Energy equation:} Multiplying the above equation by the stream function $\psi$ and integrating in space we get
\begin{equation*}
    \left<\psi\partial_t\zeta\right>+\left<\psi J(\psi,\zeta+\beta y)\right>=\left<\nu\psi\nabla^2\nabla^2\psi\right>,
\end{equation*}
where 
\begin{equation*}
    \left<\cdot\right>=\frac{1}{l^2}\iint \cdot\ dxdy.
\end{equation*}
\begin{rem}
\begin{itemize}
    \item [1)] $div(f\nabla v)=\nabla f \cdot \nabla v + f \nabla^2v $;  $f,v:\mathbb{R}^n\to\mathbb{R}$.
    \item[2)] $J(\psi,\psi(\zeta+\beta y))=\psi J(\psi,\zeta+\beta y)$. \\
    In fact, 
    \begin{align*}
        J(\psi,\psi(\zeta+\beta y))&=u(\psi(\zeta+\beta y))_x+v(\psi(\zeta+\beta y))_y\\
        &=u\psi_x(\zeta+\beta y)+u\psi(\zeta+\beta y)_x+v\psi_y(\zeta+\beta y)+v\psi(\zeta+\beta y)_y\\
        &=J(\psi,(\zeta+\beta y))\psi.
    \end{align*}
\end{itemize}
\end{rem}
Then, 
\begin{align}\label{energyrelation}
    \oint \psi \nabla\psi_t\cdot \widehat{n} dl-\left<\nabla\psi\cdot\nabla\psi_t\right>+\left<J(\right.&\left.\psi,\psi(\zeta+\beta y))\right>\\
    &=\nu\oint\psi\nabla\zeta\cdot\widehat{n} dl-\left<\nu\nabla\psi\cdot\nabla\nabla^2\psi\right>.\notag
\end{align}
Since we have periodic boundary conditions and incompressibility, equation \eqref{energyrelation} becomes
\begin{equation*}
    \partial_t\left(\frac{1}{2}\left<|\nabla\psi|^2\right>\right)=-\nu\left<\psi_{xx}^2+2\psi_{xy}^2+\psi_{yy}^2\right>.
\end{equation*}
In particular 
\begin{equation*}
    \partial_t\left(\frac{1}{2}\left<|\nabla\psi|^2\right>\right)\leq 0,
\end{equation*}
so that 
\begin{equation*}
    \left<|\nabla\psi|^2\right>\leq \left<|\nabla\psi_0|^2\right>.
\end{equation*}
We can now separate mean and eddy contributions:
\begin{align}
\left<|\nabla\psi|^2\right>&=\left<u^2+v^2\right>=\left<(\overline{u}+u')^2+v'^2\right>\notag\\
    &=\left<\overline{u}^2+2u'\overline{u}+u'^2+v'^2\right>\notag\\
    &=\left<\overline{u}^2\right>+\left<u'^2+v'^2\right>\leq \left<|\nabla\psi_0|^2\right>.\label{energy}
\end{align}
This implies that both $\left<\overline{u}^2\right>$ and $\left<u'^2+v'^2\right>$ are bounded.

\textbf{Enstrophy equation:} Similarly to the previous case, we have 
\begin{equation*}
    \left<\zeta\zeta_t\right>+\left<\zeta J(\psi,\zeta+\beta y)\right>=\nu\left<\zeta\nabla^2\zeta\right>
\end{equation*}
\begin{equation*}
    \Rightarrow \partial_t\left(\frac{1}{2}\left<\zeta^2\right>\right)=-\nu\left<|\nabla\zeta|^2\right>.
\end{equation*}
Then,
\begin{equation*}
    \left<\zeta^2\right>\leq\left<\zeta_0^2\right>.
\end{equation*}
Separating mean and eddy contributions 
\begin{align}
    \left<\zeta^2\right>&=\left<(-\overline{u}_y+\zeta')^2\right>\notag\\
    &=\left<\overline{u}^2_y-2\overline{u}_y\zeta'+\zeta'^2\right>\notag\\
    &=\left<\overline{u}_y^2\right>+\left<\zeta'^2\right>\leq\left<\zeta_0^2\right> \label{enstrophy}
\end{align}
and therefore $\left<\zeta'^2\right>$ is also bounded. 

We have thus proved
\begin{prop}\label{initialdatabound}
Let $\left(\mu_0(y), \gamma_0(x,y)\right)\in H^s(\mathbb{T}_l)\cross H^s(\mathbb{T}^2_l)$ be the initial data of problem \eqref{sistemariscritto}. Then, using notation \eqref{notatmugamma}, we have
\begin{equation}
    \norm{\gamma(t)}_{L^2_{xy}}\leq \norm{\gamma_0}_{L^2_{xy}} \text{ and }\ \norm{\mu(t)}_{L^2_{y}}\lesssim\norm{\mu_0}_{L^2_{y}} \ \ \ \ \forall \ t\geq 0,
\end{equation}
with $C>0$.
\end{prop}

\subsection{Global Well-Posedness}
In this section we complete the proof of Theorem \ref{theoprincipale}. In Section 3, we established a priori upper bounds for the solution over a finite time interval and proved local well-posedness in time and space using the contraction theorem. We also obtained spatial extension of the solution through Proposition \ref{estensionesoluzionesystem}. The final step to be addressed is time extension.

More specifically, we proved that the problem \eqref{sistemariscritto} has a unique solution on 
\begin{equation*}
    {X^{s,\delta}}=\left\{(f,g)\in L^\infty_\delta(H^s_y\times H^s_{xy});\widehat{f}(0)=0,\ \widehat{g}(0)=0\right\},
\end{equation*}
with 
\begin{equation*}
    s\geq 0 \text{ and } \delta\sim(\norm{\gamma_0}_{H^s_{xy}}+\norm{\mu_0}_{H^s_y})^{-\frac{1}{1-\alpha}}\ \ \alpha\in\left(\frac{3}{4},1\right),
\end{equation*}
which we can explicitly write as
\begin{align}
    (\mu,\gamma)=\bigg(e^{\nu t\partial_{yy}}&\mu_0+\int_0^t e^{(-\nu\partial_{yy})(t-t')}G(\gamma)\ dt',\notag\\
    e^{\tilde{D}^2t}&\left.\gamma_0+\int_0^t e^{\tilde{D}^2(t-t')}F(\mu,\gamma)\ dt'\right)\label{equafin}
\end{align}
where $\tilde{D}^2=\nu\nabla^2-\partial_x\nabla^{-2}$. We are now ready to close the proof of Theorem \ref{theoprincipale}.

\begin{proof}[Proof of Theorem \ref{theoprincipale}]
We start by choosing  
\begin{equation}\label{deltariscritto}
    \delta\sim(\norm{\gamma_0}_{L^2_{xy}}+\norm{\mu_0}_{L^2_y})^{-\frac{1}{1-\alpha}}
\end{equation}
instead of \eqref{explicitdelta}.  
By previous computations we have existence and uniqueness of the solution to \eqref{sistemariscritto} on $X^{s,\delta}$ with $\delta$ as in \eqref{deltariscritto}.
Moreover, by Proposition \ref{initialdatabound}, we have 
\begin{equation}\label{normalimitata}
    \norm{\gamma(t)}_{L_{xy}^2}+\norm{\mu(t)}_{L_y^2}\leq \left(\norm{\gamma_0}_{L^2_{xy}}+C\norm{\mu_0}_{L_y^2}\right)\ \ \ \forall\ t\geq 0.
\end{equation}
Then, if we set as initial data
\begin{equation*}
    \left(\mu(\delta),\gamma(\delta)\right),
\end{equation*}
we can use the same techniques as in previous sections to establish the existence and uniqueness of a solution on $X^{s,[\delta,\delta+\delta^*]}$ with a new parameter, which we call $\delta^*$, of the following form 
\begin{equation}\label{deltastar}
    \delta^*\sim(\norm{\gamma(\delta)}_{L^2_{xy}}+\norm{\mu(\delta)}_{L^2_y})^{-\frac{1}{1-\alpha}}.
\end{equation}
Due to \eqref{normalimitata}, 
\begin{equation*}
    \delta\leq \delta^*,
\end{equation*}
and in particular we have well-posedness of \eqref{sistemariscritto} on $X^{s,[\delta,2\delta]}$. Iterating in this way, we obtain the solution on $X^{s}$.
\end{proof}


\subsection*{Acknowledgment}
The author would like to thank G. Staffilani and R. Ferrari for suggesting the problem and for many useful conversations during the preparation of this paper. The author would also like to thank MIT for its hospitality.

\end{document}